\documentclass[reqno,12pt]{amsart}
\usepackage{amssymb}
\usepackage{mathrsfs}
\usepackage{amsmath,amssymb}
\usepackage{paralist}
\usepackage{graphics} %% add this and next lines if pictures should be in esp format
\usepackage{epsfig} %For pictures: screened artwork should be set up with an 85 or 100 line screen
\usepackage[colorlinks = true]{hyperref}
\hypersetup{urlcolor = red, citecolor = blue}
\usepackage[top = 1in,bottom = 1in,left = 1in,right = 1in]{geometry}
\usepackage{cite}
\usepackage{lineno}
\newtheorem{thm}{Theorem}[section]

\newtheorem{lem}[thm]{Lemma}

\theoremstyle{definition}
\newtheorem{defn}{Definition}[section]

\numberwithin{equation}{section}
\def\d{\,\mathrm{d}}
\def\dt{\frac{\mathrm{d}}{\mathrm{d}t}}
\def\r{\!\!\!}

\allowdisplaybreaks[3]

\begin{document}

%\linenumbers  %添加行号
%\let\oldalign\align
%\let\oldendalign\endalign
%\renewenvironment{align}{\linenomathNonumbers\oldalign}{\oldendalign\endlinenomath}
%\pagewiselinenumbers  %每页重新开始添加行号

%\title[chemotaxis system with singular density-suppressed motility and superlinear consumption]
%{Global solutions to a chemotaxis system with singular density-suppressed motility and superlinear consumption}
\title[ZHANG AND LI] 
{Global weak solutions to a two-dimensional doubly degenerate nutrient taxis system with logistic source}
\author[Zhang]{Zhiguang Zhang }%
\address{School of Mathematics, Southeast University, Nanjing 211189, P. R. China; and School of Mathematics and Statistics, Chongqing Three Gorges University, Wanzhou 404020 , P. R. China}
\email{guangz$\_$z@163.com}

\author[Li]{Yuxiang Li$^{\star}$}
\address{School of Mathematics, Southeast University, Nanjing 211189, P. R. China}
\email{lieyx@seu.edu.cn}

\thanks{$^{\star}$Corresponding author.}
\thanks{Supported in part by National Natural Science Foundation of China (No. 12271092, No. 11671079), Jiangsu Provincial Scientific Research Center of Applied Mathematics (No. BK20233002) and Science and Technology Research Program of Chongqing Municipal Education Commission (No. KJQN202201226).}



% \date{}%
% \dedicatory{}%
% \commby{}%
% ----------------------------------------------------------------
\begin{abstract}
In this work, we study the doubly degenerate nutrient taxis system with logistic source
\begin{align}
\begin{cases}\tag{$\star$}\label{eq 0.1} 
u_t=\nabla \cdot(u^{l-1} v \nabla u)- \nabla \cdot\left(u^{l} v \nabla v\right)+ u - u^2, \\ v_t=\Delta v-u v 
\end{cases}
\end{align} 
in a smooth bounded domain $\Omega \subset \mathbb{R}^2$, where $l \geqslant 1$.  It is proved that for all reasonably regular initial data, the corresponding homogeneous Neumann initial-boundary value problem \eqref{eq 0.1} possesses a global weak solution which is continuous in its first and essentially smooth in its second component. We point out that when $l = 2$, our result is consistent with that of [G. Li and M. Winkler, Analysis and Applications, (2024)].
\end{abstract}
\maketitle

\section{Introduction and main results}\label{section1}
In the present work, we consider a doubly degenerate nutrient taxis system with logistic source
\begin{equation}\label{SYS:MAIN}
\begin{cases}u_t=\nabla \cdot(u^{l-1} v \nabla u)- \nabla \cdot\left(u^{l}v \nabla v\right)+  u - u^2, & x \in \Omega, t>0, \\ v_t=\Delta v-u v, & x \in \Omega, t>0, \\ \left(u v \nabla u- u^{2} v \nabla v\right) \cdot \nu=\nabla v \cdot \nu=0, & x \in \partial \Omega, t>0, \\ u(x, 0)=u_0(x), \quad v(x, 0)=v_0(x), & x \in \Omega, \end{cases}
\end{equation}
where $\Omega \subset \mathbb{R}^2$ is a bounded domain with smooth boundary, $\nu$ is the outward unit normal vector at a point of $\partial \Omega$ and $l \geqslant 1$. The scalar functions $u$ and $v$ denote the concentration of the nutrient, respectively.

The experimental results documented by Fujikawa \cite{1992-PASMaiA-Fujikawa}, Fujikawa and Matsushita \cite{1989-JotPSoJ-FujikawaMatsushita}, as well as Matsushita and Fujikawa \cite{1990-PASMaiA-MatsushitaFujikawa} provided compelling evidence of the nuanced collective dynamics exhibited by \textit{Bacillus subtilis} populations under conditions of nutrient scarcity; in particular, evolution into complex patterns, and even snowflake-like
population distributions appear as a generic feature rather than an exceptional event in such frameworks. Following experimental evidence indicating limitations in bacterial motility near regions of low nutrient concentration in the relatively soft agar plates, as a mathematical description for such processes, Kawasaki et al. \cite{1997-JoTB-Kawa} proposed the doubly degenerate parabolic system
\begin{align}\label{sys:Kaw}
\begin{cases}
u_t=\nabla \cdot (D_u(u,v) \nabla u)+  u v, \\ 
v_t=D_v \Delta v-  u v,
\end{cases}
\end{align}
where $D_u(u,v)$ is the diffusion coefficient of the bacterial cells, $D_v$ is the constant diffusion coefficient of the nutrients. Evidences presumably suggest that bacteria are immotile when either $u$ or $v$ is low and become active as $u$ or $v$
increases, so they proposed the simplest diffusion coefficient as 
\begin{align*}
D_u(u,v)=uv.
\end{align*} 
This system associated with no-flux boundary condition in a smoothly bounded convex domain $\Omega \subset \mathbb{R}^n$ was studied by Winkler \cite{2022-CVPDE-Winklera} and was proved to be exhibit stabilization of arbitrary structures under some assumptions on initial data.

To describe the formation of the said aggregation pattern with more accuracy, Leyva et al. \cite{2013-PA-LeyvaMalagaPlaza} extended the degenerate diffusion model \eqref{sys:Kaw} to the following doubly degenerate nutrient taxis system 
\begin{align}\label{SYS:LEYVA}
\begin{cases}
u_t=\nabla \cdot(u v \nabla u)- \nabla \cdot\left(u^{2}v \nabla v\right)+  u v, \\ 
v_t=\Delta v- u v,  
\end{cases}
\end{align}
In the two dimensional setting, numerical simulations in \cite{1997-JoTB-Kawa, 2000-AiP-Ben-JacobCohenLevine, 2013-PA-LeyvaMalagaPlaza} indicated that, depending on the initial data and parameter conditions, the model \eqref{SYS:LEYVA} could generate rich branching pattern, which is very close to that observed in biological experiments.
In \cite{2021-TAMS-Winkler}, Winkler studied \eqref{SYS:LEYVA} in the one-dimensional setting, i.e. the following cross-diffusion system 
\begin{align}\label{ONESYS:WINKLER}
\begin{cases}
u_t=\left(u v u_x\right)_x-\left(u^2 v v_x\right)_x+u v, & x \in \Omega, t>0, \\ 
v_t=v_{x x}-u v, & x \in \Omega, t>0, \\ 
u v u_x-u^2 v v_x=0, \quad v_x=0, & x \in \partial \Omega, t>0, \\ 
u(x, 0)=u_0(x), \quad v(x, 0)=v_0(x), & x \in \Omega,
\end{cases}
\end{align}
where the initial data in \eqref{ONESYS:WINKLER} are assumed to be such that 
\begin{align}\label{initial-1}
\begin{cases}
u_0 \in C^{\vartheta}(\overline{\Omega}) \text { for some } \vartheta \in(0,1), \text { with } u_0 \geqslant 0 \text { and } \int_{\Omega} \ln u_0>-\infty, \quad \text { and that } \\ 
v_0 \in W^{1, \infty}(\Omega) \text { satisfies } v_0>0 \text { in } \overline{\Omega}.
\end{cases}
\end{align}\label{0930-1047}
Based on the method of energy estimates, the author demonstrated that the system \eqref{ONESYS:WINKLER} admits a global weak solution, which is uniform-in-time bounded and converges to some equilibrium within a defined topological space. Later, Li and Winkler \cite{2022-CPAA-LiWinklera} removed $\int_{\Omega} \ln u_0>-\infty$ in \eqref{initial-1} and obtained similar results. In the higher dimensional setting, for a variant of  \eqref{SYS:LEYVA}
\begin{align}
\begin{cases}u_t=\nabla \cdot(u v \nabla u)-\nabla \cdot\left(u^\alpha v \nabla v\right)+ u v, \\ v_t=\Delta v-u v,\end{cases}
\end{align}
in a smooth bounded convex domain $\Omega \subset \mathbb{R}^n$, with zero-flux boundary conditions, where $\alpha>0$. It is shown in \cite{2022-JDE-Li}, the system admits a global weak solution when either $\alpha \in \left(1, \frac{3}{2}\right)$ and $n=2$ or $\alpha \in \left(\frac{7}{6}, \frac{13}{9}\right)$ and $n=3$. Recently, another similar variant of \eqref{SYS:LEYVA}
\begin{align*}
\begin{cases}u_t=\nabla \cdot(u v \nabla u)-\nabla \cdot(u^2 v \nabla v)+\rho u-\mu u^\kappa, \\ v_t=\Delta v-u v,\end{cases}
\end{align*} 
was proposed in \cite{2024-NARWA-Pan} to reveal that the global existence of weak solutions was proven in the higher dimensional setting when $\kappa>(n+2)/2$ with $\rho, \mu>0$ and the initial data are reasonably regular and arbitrary large. In the case $\kappa=2$, Li and Winkler \cite{2024-AaA-LiWinkler} proved that the model possesses global continuous weak solutions for all reasonably regular initial data when $n=2$. Very recently, Winkler \cite{2024-JDE-Winkler} considered a more general variant of the system \eqref{SYS:LEYVA}, including the system \eqref{0930-1047} as a special case. Global bounded weak solutions were obtained in a bounded convex planar domain when either $\alpha<2$ and the initial data are reasonably regular as well as arbitrary large, or $\alpha=2$ but the initial data are such that $v_0$ satisfies an appropriate smallness condition.

Let us point out that, in the past decade, the following relatively simpler nutrient taxis system      
\begin{equation}\label{sys-1.5}
\begin{cases}
u_t=\Delta u-\nabla \cdot(u \nabla v), \\ v_t=\Delta v-u v,
\end{cases}
\end{equation}
has been studied extensively by many authors. It is proved that the associated no-flux  initial-boundary value problem for the system \eqref{sys-1.5} admits global classical solutions in bounded planar domains \cite{2012-CPDE-Winkler}, and admits global weak solutions in bounded three-dimensional domains that eventually become smooth in the large time \cite{2012-jde-taoyou}. In the higher dimensional setting, Tao \cite{2011-JMAA-Tao} obtained global bounded classical solutions under a smallness assumption on the initial signal concentration $v_0$. Wang and Li \cite{2019-EJDE-WangLi} showed that this model possesses at least one global renormalized solution. For more related results, we refer readers to the surveys \cite{2015-MMMAS-BellomoBellouquidTaoWinkler,2023-SAM-LankeitWinkler} and the references therein.
%and concerning works in a fluid-context we refer to \cite{2023-JEMSJ-Winkler,2013-AIHPCANL-TaoWinkler,2014-ARMA-Winkler,2016-AIHPCANL-Winkler,2016-CVPDE-CaoLankeit,2017-TAMS-Winkler} and for further reading. 
Recently, blow-up problems for nutrient taxis system have been investigated by some authors. We refer the reader to \cite{2024-JNS-Jin} for singular attraction-consumption systems, and to \cite{2023-PRSESA-WangWinkler, 2023-CVPDE-AhnWinkler} for repulsion-consumption systems.

Motivated by the aforementioned works, this paper aims to investigate the global existence of weak solutions for a doubly degenerate nutrient taxi system \eqref{SYS:MAIN} with a logistic source in a two-dimensional setting.

\vskip 3mm

\textbf{Main results.} 
Throughout our analysis, the initial data in \eqref{SYS:MAIN} will be assumed
to be such that
\begin{align}\label{assIniVal}
\begin{cases}
u_0 \in \left\{\begin{array}{lll}
W^{1, \infty}(\Omega) \text {  with  } u_0 \geqslant 0   & \text {  when  } & 1 \leqslant l<3 , \\
W^{1, \infty}(\Omega) \text { with } u_0 \geqslant 0 \text {  and  } \int_{\Omega} \ln u_0>-\infty    & \text {   when   } & l=3, \\
W^{1, \infty}(\Omega) \text {   with  } u_0 \geqslant 0 \text {  and   } \int_{\Omega} u_{0}^{3-l}<\infty    & \text {   when   } & l>3  \quad \text {  as well as   }  \\  
\end{array}\right.\\ v_0 \in W^{1, \infty}(\Omega) \text {    satisfies   }  v_0 > 0  \text { in } \overline{\Omega}.
\end{cases}
\end{align}

%%%%%%%%%%%%%%%%%%%%%%%%%%%%%%%%%%%%%%%%%%%%%%%%%%%%%%%%%%%%%%%%%%%%%%%%%%%%%%%%%%%%%%%%
\begin{defn} \label{def-weak-sol}
Let $l\geqslant1$ and let $\Omega \subset \mathbb{R}^2$ be a bounded domain with smooth boundary. Suppose that $u_0 \in L^1(\Omega)$ and $v_0 \in L^1(\Omega)$ are nonnegative. By a global weak solution of the system (\ref{SYS:MAIN}) we mean a pair $(u, v)$ of functions satisfying
\begin{equation}\label{-2.1}
\begin{cases}
u \in L_{\mathrm{loc}}^1(\overline{\Omega} \times[0, \infty)) \quad \text { and } \\
v \in L_{\mathrm{loc}}^{\infty}(\overline{\Omega} \times[0, \infty)) \cap L_{\mathrm{loc}}^1\left([0, \infty) ; W^{1,1}(\Omega)\right)
\end{cases}
\end{equation}
and
\begin{equation}\label{-2.2}
u^l \in L_{\mathrm{loc}}^1\left([0, \infty) ; W^{1,1}(\Omega)\right) \quad \text { and } \quad u^{l} \nabla v \in L_{\mathrm{loc}}^1\left(\overline{\Omega} \times[0, \infty) ; \mathbb{R}^2\right),
\end{equation}
which are such that
\begin{align}\label{-2.3}
-\int_0^{\infty} \int_{\Omega} u \varphi_t-\int_{\Omega} u_0 \varphi(\cdot, 0)=&-\frac{1}{l} \int_0^{\infty} \int_{\Omega} v  \nabla u^{l} \cdot \nabla \varphi
+ \int_0^{\infty} \int_{\Omega} u^{l} v \nabla v \cdot \nabla \varphi \nonumber\\
& + \int_0^{\infty} \int_{\Omega} u  \varphi - \int_0^{\infty} \int_{\Omega} u^2  \varphi
\end{align}
and
\begin{align}\label{-2.4}
\int_0^{\infty} \int_{\Omega} v \varphi_t+\int_{\Omega} v_0 \varphi(\cdot, 0)=\int_0^{\infty} \int_{\Omega} \nabla v \cdot \nabla \varphi+\int_0^{\infty} \int_{\Omega} u v \varphi
\end{align}
for all $\varphi \in C_0^{\infty}(\overline{\Omega} \times[0, \infty))$.
\end{defn}
%%%%%%%%%%%%%%%%%%%%%%%%%%%%%%%%%%%%%%%%%%%%%%%%%%%%%%%%%%%%%%%%%%%%%%%%%%%%%%%%%%%%%%%%

\begin{thm} \label{thm-1.1}
Let $l\geqslant1$ and let $\Omega \subset \mathbb{R}^2$ be a bounded domain with smooth boundary. Assume that the initial data $\left(u_0, v_0\right)$ satisfies \eqref{assIniVal}. Then there exist functions
\begin{align}\label{solu:property}
\begin{cases}
u \in C^{0}(\overline{\Omega} \times[0, \infty)) \quad \text { and } \\
v \in C^0(\overline{\Omega} \times[0, \infty)) \cap C^{2,1}(\overline{\Omega} \times(0, \infty))
\end{cases}
\end{align}
such that $u \geqslant 0$ in $\Omega \times(0, \infty)$ and $v>0$ in $\overline{\Omega} \times[0, \infty)$, and that $(u, v)$ solves \eqref{SYS:MAIN} in the sense of Definition \ref{def-weak-sol}.
\end{thm}

\vskip 1mm

The remainder of this paper is organized as follows. 
In Section \ref{section2}, we define the global weak solutions of the system \eqref{SYS:MAIN} and provide some initial findings regarding the local-in-time existence of the problem \eqref{SYS:MAIN}. In Section \ref{sect-3}, we give $L^p$ bounds for $u_{\varepsilon}$ and state straightforward consequences.
In Section \ref{section4}, the proof of the Theorem \ref{thm-1.1} is proven.

%%%%%%%%%%%%%%%%%%%%%%%%%%%%%%%%%%%%%%%%%%%%%%%%%%%%%%%%%%%%%%%%%%%%%%%%%%%%%%%%%%%%%%%%
%By gaining information on the corresponding dissipation rates, we can confirm that with additional constraints on $m$, $n$ and $l$ and singular behavior of $\phi$, global solvability can be ensured in frameworks of standard weak solvability with respect to the first equation in (\ref{SYS:MAIN}).
%%%%%%%%%%%%%%%%%%%%%%%%%%%%%%%%%%%%%%%%%%%%%%%%%%%%%%%%%%%%%%%%%%%%%%%%%%%%%%%%%%%%%%%%

%%%%%%%%%%%%%%%%%%%%%%%%%%%%%%%%%%%%%%%%%%%%%%%%%%%%%%%%%%%%%%%%%%%%%%%%%%%%%%%%%%%%%%%%

\section{Preliminary estimates for the regularized problems}\label{section2}

In order to construct a weak solution of the system \eqref{SYS:MAIN}, we consider the following regularized problem for $l \geqslant 1$ and $\varepsilon \in(0,1)$
\begin{align}\label{sys-regul}
\begin{cases}
u_{\varepsilon t}=\nabla \cdot\left(u^{l-1}_{\varepsilon} v_{\varepsilon} \nabla u_{\varepsilon}\right)
  - \nabla \cdot\left(u^{l}_{\varepsilon} v_{\varepsilon} \nabla v_{\varepsilon}\right)+ u_{\varepsilon} - u_{\varepsilon}^2, 
    & x \in \Omega, t>0, \\ 
v_{\varepsilon t}=\Delta v_{\varepsilon}-u_{\varepsilon} v_{\varepsilon}, 
   & x \in \Omega, t>0, \\ 
\frac{\partial u_{\varepsilon}}{\partial \nu}=\frac{\partial v_{\varepsilon}}{\partial \nu}=0, 
   & x \in \partial \Omega, t>0, \\ 
u_{\varepsilon}(x, 0)=u_{0 \varepsilon}(x)=u_0(x)+\varepsilon, \quad 
   v_{\varepsilon}(x, 0)=v_{0 \varepsilon}(x)=v_0(x), 
   & x \in \Omega.
\end{cases}
\end{align}
 
We first give the following standard local existence and extensibility criterion theorem.

\begin{lem}\label{lemma-2.1}
Let $l \geqslant 1$ and $\Omega \subset \mathbb{R}^2$ be a bounded domain with smooth boundary, and suppose that (\ref{assIniVal}) holds. Then for each $\varepsilon \in(0,1)$, there exist $T_{\max, \varepsilon} \in(0, \infty]$ and at least one pair $\left(u_{\varepsilon}, v_{\varepsilon}\right)$ of functions
\begin{align}\label{-2.6}
\begin{cases}
u_{\varepsilon} \in C^0\left(\overline{\Omega} \times\left[0, T_{\max, \varepsilon}\right)\right) \cap C^{2,1}\left(\overline{\Omega} \times\left(0, T_{\max, \varepsilon}\right)\right) \quad \text{ and } \\
v_{\varepsilon} \in \cap_{q > 2} C^0\left(\left[0, T_{\max, \varepsilon}\right); W^{1, q}(\Omega)\right) \cap C^{2,1}\left(\overline{\Omega} \times\left(0, T_{\max, \varepsilon}\right)\right)
\end{cases}
\end{align}
which are such that $u_{\varepsilon}>0$ and $v_{\varepsilon}>0$ in $\overline{\Omega} \times\left[0, T_{\max, \varepsilon}\right)$, that $\left(u_{\varepsilon}, v_{\varepsilon}\right)$ solves \eqref{sys-regul} in the classical sense, and that
\begin{align}\label{-2.7}
if \text{  } T_{\max, \varepsilon}<\infty, \quad \text { then } \quad \limsup _{t \nearrow T_{\max, \varepsilon}}\left\|u_{\varepsilon}(t)\right\|_{L^{\infty}(\Omega)}=\infty.
\end{align}
In addition, this solution satisfies
\begin{align}\label{-2.9}
\left\|v_{\varepsilon}(t)\right\|_{L^{\infty}(\Omega)} \leqslant \|v_0\|_{L^{\infty}(\Omega)}, \quad \text { for all } t \in\left(0, T_{\max, \varepsilon}\right)
\end{align}
and
\begin{align}\label{-2.10}
\int_{0}^{T_{\max, \varepsilon}}\!\!\! \int_{\Omega} u_{\varepsilon} v_{\varepsilon} \leqslant \int_{\Omega} v_0.
\end{align}
\end{lem}
\begin{proof}
This lemma can be proved using Amann’s theory \cite{Amann-1989}, and for details we refer to
\cite[Lemma 2.1]{2022-NARWA-Winkler}. 
\end{proof}

Next, we will give the basic properties of solutions to the approximate problem
\eqref{sys-regul}.

%%%%%%%%%%%%%%%%%%%%%%%%%%%%%%%%%%%%%%%%%%%%%%%%%%%%%%%%%%%%%%%%%%%%%%%%%%%%%

\begin{lem}\label{lem-1st-est}
Let $l\geqslant1$ and $T = \min \{\widetilde{T} , T_{\max,\varepsilon}\}$ for $\widetilde{T} \in \left(0, +\infty\right)$, and assume that \eqref{assIniVal} holds. Then there holds
\begin{align}\label{-3.4}
\int_{\Omega} u_{\varepsilon} \leqslant m_*= \int_{\Omega} (u_{0}+1) \quad \text { for all } \varepsilon \in(0,1).
\end{align}
\end{lem}

%%%%%%%%%%%%%%%%%%%%%%%%%%%%%%%%%%%%%%%%%%%%%%%%%%%%%%%%%%%%%%%%%%%%%%%%%%%%%
\begin{proof}
By Hölder's inequality, $\left(\int_{\Omega} u_{\varepsilon}\right)^2 \leqslant \left(\int_{\Omega} u_{\varepsilon}^2\right)|\Omega|$. Using the first equation of \eqref{sys-regul} and integrating by parts, we obtain
\begin{align}\label{0919-2254}
\dt \int_{\Omega} u_{\varepsilon} & = \int_{\Omega} u_{\varepsilon}-\int_{\Omega} u_{\varepsilon}^2 \nonumber\\
& \leqslant  \int_{\Omega} u_{\varepsilon}-\frac{1}{|\Omega|}\left(\int_{\Omega} u_{\varepsilon}\right)^2 \quad \text { for all } t \in(0, T) \text { and } \varepsilon \in(0,1).
\end{align}
The claim can be obtained by solving the ODE equation.
\end{proof}

\begin{lem}\label{lem-1st-est1}
Let $l\geqslant1$ and $T = \min \{\widetilde{T} , T_{\max,\varepsilon}\}$ for $\widetilde{T} \in \left(0, +\infty\right)$, and assume that \eqref{assIniVal} holds. Then there exists $C(T)>0$ such that 
\begin{align}\label{-3.5}
\int_0^T \int_{\Omega} u^2_{\varepsilon}\leqslant C(T)   \quad \text { for all } \varepsilon \in(0,1),
\end{align}
where $C(T)$ is a positive constant depending on  $\int_{\Omega} u_0$, but independent of $\varepsilon$.
\end{lem}
\begin{proof}
The estimate
\begin{align*}
\int_0^T \int_{\Omega} u_{\varepsilon}^2& \leqslant  \int_0^T \int_{\Omega} u_{\varepsilon}+ \int_{\Omega} u_{0 \varepsilon}- \int_{\Omega} u_{\varepsilon}(T) \\
& \leqslant  (T + 1) m_* \quad \text { for all } t \in(0, T) \text { and } \varepsilon \in(0,1),
\end{align*}
results from \eqref{0919-2254} after time-integration.
\end{proof}

We use the estimate \eqref{-3.5} to derive the following additional
information in the second solution component.
\begin{lem}\label{lem-1st-est2}
Let $l\geqslant1$ and $T = \min \{\widetilde{T} , T_{\max,\varepsilon}\}$ for $\widetilde{T} \in \left(0, +\infty\right)$, and assume that \eqref{assIniVal} holds. Then there exists $C(T)>0$ such that 
\begin{align}\label{-3.7}
\int_0^T \int_{\Omega}\left| v_{\varepsilon t }\right|^2  \leqslant C(T)   \quad \text { for all } \varepsilon \in(0,1),
\end{align}
where $C(T)$ is a positive constant depending on $\int_{\Omega} u_0$ and $\int_{\Omega} |\nabla v_{0 }|^2$, but independent of $\varepsilon$.
\end{lem}
\begin{proof}
Multiplying the second equation of \eqref{sys-regul} by $-\Delta v_{\varepsilon}$ and integrating over $\Omega$ yield that
\begin{align*}
\frac{1}{2} \frac{d}{d t} \int_{\Omega}|\nabla v_{\varepsilon}|^2+\int_{\Omega}|\Delta v_{\varepsilon}|^2=\int_{\Omega}u_{\varepsilon} v_{\varepsilon}  \Delta v_{\varepsilon} \quad \text { for all } t \in(0, T) \text { and } \varepsilon \in(0,1).
\end{align*}
This together with \eqref{-2.9} and Young's inequality gives
\begin{align*}
\frac{1}{2} \frac{d}{d t} \int_{\Omega}|\nabla v_{\varepsilon}|^2+\int_{\Omega}|\Delta v_{\varepsilon}|^2 \leqslant \frac{1}{2} \int_{\Omega}|\Delta v_{\varepsilon}|^2+\frac{1}{2}\left\|v_0\right\|_{L^{\infty}(\Omega)}^2 \int_{\Omega} u_{\varepsilon}^2
\end{align*}
for all $t \in(0, T)$ and $\varepsilon \in(0,1)$. Integrating the above inequality on $(0, T)$, we obtian 
\begin{align}\label{2-2.54}
\int_{\Omega}\left|\nabla v_{\varepsilon}(T)\right|^2+\int_0^T \int_{\Omega}\left|\Delta v_{\varepsilon}\right|^2 \leqslant \int_{\Omega}\left|\nabla v_0\right|^2+\left\|v_0\right\|_{L^{\infty}(\Omega)}^2 \int_0^T \int_{\Omega} u_{\varepsilon}^2
\end{align}
for all $t \in(0, T)$ and $\varepsilon \in(0,1)$. Since $(a+b)^2 \leqslant 2(a^2+b^2)$ for all $a, b > 0$, we can estimate
\begin{align*}
v_{\varepsilon t}^2 \leqslant \left(\left|\Delta v_{\varepsilon}\right|+u_{\varepsilon} v_{\varepsilon}\right)^2 \leqslant 2 \left|\Delta v_{\varepsilon}\right|^2+2\left\|v_0\right\|_{L^{\infty}(\Omega)}^2 u_{\varepsilon}^2\quad \text { for all } t \in(0, T) \text { and } \varepsilon \in(0,1). 
\end{align*}
Combining this with \eqref{-3.5} and \eqref{2-2.54}, we complete the proof of \eqref{-3.7}.
\end{proof}

%%%%%%%%%%%%%%%%%%%%%%%%%%%%%%%%%%%%%%%%%%%%%%%%%%%%%%%%%%%%%%%%%%%%%%%%%%%%%

\begin{lem}\label{lem-2nd-est}
Let $l\geqslant1$ and $T = \min \{\widetilde{T} , T_{\max,\varepsilon}\}$ for $\widetilde{T} \in \left(0, +\infty\right)$, and assume that \eqref{assIniVal} holds. Then there exists $C(T)>0$ such that
\begin{align}\label{-3.10}
\int_0^{T}\r\int_{\Omega} \frac{|\nabla v_{\varepsilon }|^4}{v_{\varepsilon}^3} \leqslant C(T) \quad \text { for all } \varepsilon \in(0,1),
\end{align}
where $C(T)$ is a positive constant depending on $\int_{\Omega} u_0$, $\int_{\Omega} |\nabla v_{0 }|^2$ and $\int_{\Omega} \frac{|\nabla v_{0 }|^2}{v_0}$, but independent of $\varepsilon$. 
\end{lem}

%%%%%%%%%%%%%%%%%%%%%%%%%%%%%%%%%%%%%%%%%%%%%%%%%%%%%%%%%%%%%%%%%%%%%%%%%%%%%

\begin{proof}
Since $v_{\varepsilon}$ is positive in $\overline{\Omega} \times(0, T)$ and actually belongs to $C^3(\overline{\Omega} \times(0, T))$ thanks to standard parabolic regularity theory, using the second equation in the system \eqref{sys-regul} we compute
\begin{align}\label{-3.11}
\frac{1}{2} \dt \int_{\Omega} \frac{\left|\nabla v_{\varepsilon}\right|^2}{v_{\varepsilon}}
= & \int_{\Omega} \frac{1}{v_{\varepsilon}} \nabla v_{\varepsilon} \cdot \nabla\left\{\Delta v_{\varepsilon}
    -u_{\varepsilon} v_{\varepsilon}\right\}
    -\frac{1}{2} \int_{\Omega} \frac{1}{v_{\varepsilon}^2}\left|\nabla v_{\varepsilon}\right|^2 \cdot\left\{\Delta v_{\varepsilon}-u_{\varepsilon} v_{\varepsilon}\right\} \nonumber\\
= & \int_{\Omega} \frac{1}{v_{\varepsilon}} \nabla v_{\varepsilon} \cdot \nabla\Delta v_{\varepsilon}
    -\int_{\Omega} \frac{u_{\varepsilon}}{v_{\varepsilon}}\left|\nabla v_{\varepsilon}\right|^2
    -\int_{\Omega} \nabla u_{\varepsilon} \cdot \nabla v_{\varepsilon}\nonumber\\
  & -\frac{1}{2} \int_{\Omega} \frac{1}{v_{\varepsilon}^2}\left|\nabla v_{\varepsilon}\right|^2 \cdot\Delta v_{\varepsilon}
    +\frac{1}{2} \int_{\Omega} \frac{u_{\varepsilon}}{v_{\varepsilon}}\left|\nabla v_{\varepsilon}\right|^2\nonumber\\
= & \int_{\Omega} \frac{1}{v_{\varepsilon}} \nabla v_{\varepsilon} \cdot \nabla\Delta v_{\varepsilon}
    -\frac{1}{2} \int_{\Omega} \frac{1}{v_{\varepsilon}^2}\left|\nabla v_{\varepsilon}\right|^2 \cdot\Delta v_{\varepsilon}\nonumber\\
  & -\frac{1}{2} \int_{\Omega} \frac{u_{\varepsilon}}{v_{\varepsilon}}\left|\nabla v_{\varepsilon}\right|^2
    -\int_{\Omega} \nabla u_{\varepsilon} \cdot \nabla v_{\varepsilon}\quad \text { for all } t \in(0, T) \text { and } \varepsilon \in(0,1). 
\end{align}
Based on \cite[Lemma~3.2]{2012-CPDE-Winkler}, we have the integral identity
\begin{align*}
\int_{\Omega} \frac{1}{v_{\varepsilon}} \nabla v_{\varepsilon} \cdot \nabla \Delta v_{\varepsilon}
-\frac{1}{2} \int_{\Omega} \frac{1}{v_{\varepsilon}^2}|\nabla v_{\varepsilon}|^2 \Delta v_{\varepsilon}
= -\int_{\Omega} v_{\varepsilon}\left|D^2 \ln v_{\varepsilon}\right|^2
  +\frac{1}{2} \int_{\partial \Omega} \frac{1}{v_{\varepsilon}} \frac{\partial|\nabla v_{\varepsilon}|^2}{\partial \nu}
\end{align*}
for all $t \in(0, T)$ and $\varepsilon \in(0,1)$. Therefore, \eqref{-3.11} becomes
\begin{align}\label{-3.11-2}
&\frac{1}{2} \dt \int_{\Omega} \frac{\left|\nabla v_{\varepsilon}\right|^2}{v_{\varepsilon}}
  +\int_{\Omega} v_{\varepsilon}\left|D^2 \ln v_{\varepsilon}\right|^2
    +\frac{1}{2} \int_{\Omega} \frac{u_{\varepsilon}}{v_{\varepsilon}}\left|\nabla v_{\varepsilon}\right|^2\nonumber\\
   =&  \frac{1}{2} \int_{\partial \Omega} \frac{1}{v_{\varepsilon}} \frac{\partial|\nabla v_{\varepsilon}|^2}{\partial \nu}
    -\int_{\Omega} \nabla u_{\varepsilon} \cdot \nabla v_{\varepsilon}\quad \text { for all } t \in(0, T) \text { and } \varepsilon \in(0,1). 
\end{align}
By \cite[Lemma 3.3]{2012-CPDE-Winkler} and \cite[Lemma 3.4]{2022-NARWA-Winkler}, there holds
\begin{equation}\label{-3.12}
\int_{\Omega} v_{\varepsilon}\left|D^2 \ln v_{\varepsilon}\right|^2 
\geqslant c_1 \int_{\Omega} \frac{\left|D^2 v_{\varepsilon}\right|^2}{v_{\varepsilon}}
  +c_1 \int_{\Omega} \frac{\left|\nabla v_{\varepsilon}\right|^4}{v_{\varepsilon}^3}
\end{equation}
for some $c_1>0$.
Combining a pointwise upper estimate for normal derivatives in \cite[Lemma 4.2]{2014-aihcn-mizoguchi} with a boundary trace embedding inequality (cf. \cite[Theorem 1, p.272]{2010--Evans}) and Young's inequality, we have 
\begin{align}\label{n-3.13}
\frac{1}{2} \int_{\partial \Omega} \frac{1}{v_{\varepsilon}} \cdot \frac{\partial\left|\nabla v_{\varepsilon}\right|^2}{\partial \nu} 
& \leqslant c_2 \int_{\partial \Omega} \frac{\left|\nabla v_{\varepsilon}\right|^2}{v_{\varepsilon}} \nonumber\\
& \leqslant c_3 \int_{\Omega}\left|\nabla\Big(\frac{\left|\nabla v_{\varepsilon}\right|^2}{v_{\varepsilon}}\Big)\right|
    +c_3 \int_{\Omega} \frac{\left|\nabla v_{\varepsilon}\right|^2}{v_{\varepsilon}} \nonumber\\
& \leqslant 2 c_3 \int_{\Omega} \frac{\left|D^2 v_{\varepsilon} \cdot \nabla v_{\varepsilon}\right|}{v_{\varepsilon}}
    +c_3 \int_{\Omega} \frac{\left|\nabla v_{\varepsilon}\right|^3}{v_{\varepsilon}^2}
    +c_3 \int_{\Omega} \frac{\left|\nabla v_{\varepsilon}\right|^2}{v_{\varepsilon}} \nonumber\\
&\leqslant \frac{c_1}{2}\int_{\Omega} \frac{\left|D^2 v_{\varepsilon}\right|^2}{v_{\varepsilon}}
  +\frac{c_1}{4} \int_{\Omega} \frac{\left|\nabla v_{\varepsilon}\right|^4}{v_{\varepsilon}^3}
  +(3 c_1^{-1} c_3^2+c_3) \int_{\Omega} \frac{\left|\nabla v_{\varepsilon}\right|^2}{v_{\varepsilon}}\nonumber\\
&\leqslant \frac{c_1}{2} \int_{\Omega} \frac{\left|D^2 v_{\varepsilon}\right|^2}{v_{\varepsilon}}
  +\frac{c_1}{2} \int_{\Omega} \frac{\left|\nabla v_{\varepsilon}\right|^4}{v_{\varepsilon}^3}
  +c_5 \int_{\Omega} v_{\varepsilon} 
\end{align}
with some constant $c_5>0$. 
Since $\left|\Delta v_{\varepsilon}\right|^2 \leqslant 2\left|D^2 v_{\varepsilon}\right|^2$, by Young's inequality, we infer that 
\begin{align}\label{-3.13}
- \int_{\Omega} \nabla u_{\varepsilon} \cdot \nabla v_{\varepsilon}  = \int_{\Omega} u_{\varepsilon} \Delta v_{\varepsilon}  & \leqslant \frac{c_1}{4} \int_{\Omega} \frac{\left|\Delta v_{\varepsilon}\right|^2}{v_{\varepsilon}}+ c_1^{-1} \int_{\Omega} u_{\varepsilon}^2 v_{\varepsilon}\nonumber\\
& \leqslant \frac{c_1}{2} \int_{\Omega} \frac{\left|D^2 v_{\varepsilon}\right|^2}{v_{\varepsilon}}+ c_1^{-1} \int_{\Omega} u_{\varepsilon}^2 v_{\varepsilon}.
\end{align}
Summing up \eqref{-3.11-2}-\eqref{-3.13} and according to \eqref{-2.9}, we conclude that with $c_6=2 c_5 |\Omega| \left\|v_0\right\|_{L^{\infty}(\Omega)}$,
\begin{align}\label{-3.14}
\dt \int_{\Omega} \frac{|\nabla v_{\varepsilon}|^2}{v_{\varepsilon}}
+c_1 \int_{\Omega} \frac{|\nabla v_{\varepsilon}|^4}{v_{\varepsilon}^3}
+\frac{1}{2} \int_{\Omega} \frac{u_{\varepsilon}}{v_{\varepsilon}} |\nabla v_{\varepsilon }|^2 
\leqslant &   2 c_1^{-1} \int_{\Omega} u_{\varepsilon}^2 v_{\varepsilon} + 2 c_5 \int_{\Omega} v_{\varepsilon} \nonumber\\
\leqslant &  2 c_1^{-1} \left\|v_0\right\|_{L^{\infty}(\Omega)} \int_{\Omega} u_{\varepsilon}^2  + c_6
\end{align}
for all $t \in(0, T)$ and $\varepsilon \in(0,1)$. Integrating \eqref{-3.14} on $(0,T)$ shows that
\begin{align*}
&\int_{\Omega} \frac{|\nabla v_{\varepsilon}(t)|^2}{v_{\varepsilon}(T)}
  +c_1 \int_0^{T} \int_{\Omega} \frac{|\nabla v_{\varepsilon}|^4}{v_{\varepsilon}^3}
  +\frac{1}{2}\int_0^{T} \int_{\Omega} \frac{u_{\varepsilon}}{v_{\varepsilon}} |\nabla v_{\varepsilon }|^2 \\
\leqslant &\int_{\Omega} \frac{|\nabla v_{0 }|^2}{v_0}
  +2 c_1^{-1} \left\|v_0\right\|_{L^{\infty}(\Omega)} \int_0^{T} \int_{\Omega} u_{\varepsilon}^2  + c_6 T \quad \text { for all } \varepsilon \in(0,1),
\end{align*}
thus \eqref{-3.10} holds thanks to Lemma \ref{lem-1st-est1}.
\end{proof}

%%%%%%%%%%%%%%%%%%%%%%%%%%%%%%%%%%%%%%%%%%%%%%%%%%%%%%%%%%%%%%%%%%%%%%%%%%%%%

\section{$L^p$ Bounds for $u_{\varepsilon}$}\label{sect-3}

In this section, we derive the estimate for $\|u_\varepsilon(t)\|_{L^p(\Omega)}$ independent of $\varepsilon$ for any $p > 1$ in $(0,T)$. 
%In the process, the key step is to obtain a uniform estimate for $\frac{|\nabla v_{\varepsilon }|^4}{v_{\varepsilon}^3}$ with respect to $\varepsilon$ in $L^\infty((0,T_{\max, \varepsilon}),L^1(\Omega))$. 
We first deduce a differential inequality for $\int_\Omega u^p_{\varepsilon}(t)$.

%%%%%%%%%%%%%%%%%%%%%%%%%%%%%%%%%%%%%%%%%%%%%%%%%%%%%%%%%%%%%%%%%%
\begin{lem}\label{lemma-3.9xx}
Let $l\geqslant1$ and $T = \min \{\widetilde{T} , T_{\max,\varepsilon}\}$ for $\widetilde{T} \in \left(0, +\infty\right)$, and assume that \eqref{assIniVal} holds. 
Then for all $p \geqslant 2$ we have 
\begin{align}\label{0704-0024}
\dt \int_{\Omega} u_{\varepsilon}^p 
+\frac{p(p-1)}{(l+p-1)^2} \int_{\Omega}\left|\nabla\left(u_{\varepsilon}^{\frac{l+p-1}{2}} v_{\varepsilon}^{\frac{1}{2}}\right)\right|^2 
\leqslant &   A \left\{\int_\Omega u_{\varepsilon}^{2(l+p-1)}v_{\varepsilon}^2\right\}^{\frac{1}{2}} \cdot\left\{\int_\Omega \frac{|\nabla v_{\varepsilon}|^4}{v_{\varepsilon}^3} \frac{1}{v_{\varepsilon}}\right\}^{\frac{1}{2}} \nonumber\\
& +  A \left\{\int_\Omega u_{\varepsilon}^{2(l+p-1)}v_{\varepsilon}^2\right\}^{\frac{1}{2}} \cdot\left\{\int_\Omega \frac{|\nabla v_{\varepsilon}|^4}{v_{\varepsilon}^3} \right\}^{\frac{1}{2}}\nonumber\\
& + p  \int_{\Omega} u_{\varepsilon}^{p}   - p  \int_{\Omega}  u_{\varepsilon}^{p+1}
\end{align}
for all $t \in(0, T)$ and $\varepsilon \in(0,1)$, where $A=\max \left\{\frac{1}{2}\frac{(p-1)p}{(l+p-1)^2}, \frac{(p-1)p}{2}\left\|v_{0} \right\|^\frac{3}{2}_{L^{\infty}(\Omega)} \right\}$.
\end{lem}
\begin{proof}
Multiplying the first equation of \eqref{sys-regul} by $u_{\varepsilon}^{p-1}$, integrating by parts and using Young's inequality, we obtain 
\begin{align}\label{0703-1903xx}
& \dt \int_{\Omega} u_{\varepsilon}^p\nonumber\\
= &\ p \int_{\Omega} u_{\varepsilon}^{p-1}\left\{\nabla \cdot(u^{l-1}_{\varepsilon} v_{\varepsilon} \nabla u_{\varepsilon})- \nabla \cdot\left(u_{\varepsilon}^{l} v_{\varepsilon} \nabla v_{\varepsilon}\right)+ u_{\varepsilon} -u^2_{\varepsilon}\right\} \nonumber\\
= &-(p-1)p \int_{\Omega} u_{\varepsilon}^{p+l-3} v_{\varepsilon}\left|\nabla u_{\varepsilon}\right|^2+ (p-1)p \int_{\Omega} u_{\varepsilon}^{p+l-2} v_{\varepsilon} \nabla u_{\varepsilon} \cdot \nabla v_{\varepsilon} + p \int_{\Omega} u_{\varepsilon}^{p} -p \int_{\Omega} u_{\varepsilon}^{p+1}\nonumber\\
\leqslant & -\frac{(p-1)p}{2} \int_{\Omega} u_{\varepsilon}^{p+l-3} v_{\varepsilon} \left|\nabla u_{\varepsilon}\right|^2 +\frac{(p-1)p}{2} \int_{\Omega} u_{\varepsilon}^{p+l-1} v_{\varepsilon}\left|\nabla v_{\varepsilon}\right|^2+p \int_{\Omega} u_{\varepsilon}^{p}-p \int_{\Omega} u_{\varepsilon}^{p+1}
\end{align}
for all $t \in(0, T)$ and $\varepsilon \in(0,1)$. Since $(a-b)^2 \geqslant \frac{1}{2} a^2-b^2$ for all $a, b \in \mathbb{R}$, we see that
\begin{align*}
\int_{\Omega} u_{\varepsilon}^{l+p-3} v_{\varepsilon}|\nabla u_{\varepsilon}|^2 
= & \frac{4}{(l+p-1)^2} \int_{\Omega}\left|v_{\varepsilon}^{\frac{1}{2}} \nabla u_{\varepsilon}^{\frac{l+p-1}{2}}\right|^2 \\
= &\frac{4}{(l+p-1)^2}  \int_{\Omega}\left|\nabla\left(u_{\varepsilon}^{\frac{l+p-1}{2}} v_{\varepsilon}^{\frac{1}{2}}\right)-\frac{1}{2} u_{\varepsilon}^{\frac{l+p-1}{2}} v_{\varepsilon}^{-\frac{1}{2}} \nabla v_{\varepsilon}\right|^2 \\
\geqslant & \frac{2}{(l+p-1)^2} \int_{\Omega}\left|\nabla\left(u_{\varepsilon}^{\frac{l+p-1}{2}} v_{\varepsilon}^{\frac{1}{2}}\right)\right|^2-\frac{1}{(l+p-1)^2} \int_{\Omega} u_{\varepsilon}^{l+p-1} v_{\varepsilon}^{-1}|\nabla v_{\varepsilon}|^2
\end{align*}
This together with \eqref{0703-1903xx} implies
\begin{align}\label{0923-1139}
\dt \int_{\Omega} u_{\varepsilon}^p +\frac{(p-1)p}{(l+p-1)^2} \int_{\Omega}\left|\nabla\left(u_{\varepsilon}^{\frac{l+p-1}{2}} v_{\varepsilon}^{\frac{1}{2}}\right)\right|^2
\leqslant  & \frac{1}{2} \frac{(p-1)p}{(l+p-1)^2}  \int_{\Omega} u_{\varepsilon}^{l+p-1} v_{\varepsilon}^{-1}|\nabla v_{\varepsilon}|^2\nonumber\\
& +\frac{(p-1)p}{2} \int_{\Omega} u_{\varepsilon}^{p+l-1} v_{\varepsilon}\left|\nabla v_{\varepsilon}\right|^2\nonumber\\
& +p \int_{\Omega} u_{\varepsilon}^{p}-p \int_{\Omega} u_{\varepsilon}^{p+1}
\end{align}
for all $t \in(0, T)$ and $\varepsilon \in(0,1)$. By H\"{o}lder's inequality and \eqref{-2.9}, we show that
\begin{align*}
\int_{\Omega} u_{\varepsilon}^{l+p-1} v_{\varepsilon}^{-1}|\nabla v_{\varepsilon}|^2 
\leqslant \left\{\int_\Omega u_{\varepsilon}^{2(l+p-1)}v_{\varepsilon}^2\right\}^{\frac{1}{2}} \cdot\left\{\int_\Omega \frac{|\nabla v_{\varepsilon}|^4}{v_{\varepsilon}^3} \frac{1}{v_{\varepsilon}}\right\}^{\frac{1}{2}}
\end{align*}
and
\begin{align*}
\int_{\Omega} u_{\varepsilon}^{l+p-1} v_{\varepsilon}|\nabla v_{\varepsilon}|^2 
\leqslant \left\|v_{0} \right\|^\frac{3}{2}_{L^{\infty}(\Omega)} \cdot \left\{\int_\Omega u_{\varepsilon}^{2(l+p-1)}v_{\varepsilon}^2\right\}^{\frac{1}{2}} \cdot\left\{\int_\Omega \frac{|\nabla v_{\varepsilon}|^4}{v_{\varepsilon}^3} \right\}^{\frac{1}{2}}.
\end{align*}
In view of the above inequalities and \eqref{0923-1139}, we complete the proof.  
\end{proof}

From \eqref{-3.10}, we know that $\int_\Omega \frac{|\nabla v_{\varepsilon}|^4}{v_{\varepsilon}^3}\in L^1(0,T)$. However, this is not sufficient to solve the differential inequality \eqref{0704-0024} for $\int_{\Omega} u_{\varepsilon}^p$. To overcome this difficulty, we shall prove $\int_\Omega \frac{|\nabla v_{\varepsilon}|^4}{v_{\varepsilon}^3}\in L^\infty(0,T)$ and $\frac{1}{v_{\varepsilon}}\in L^\infty((0,T) \times \Omega)$. To this aim, we derive a differential inequality for the following energy-like functional $G_{\varepsilon}(t)$ and the estimate of the temporal-spatial integral of $u^2_{\varepsilon} \ln^2 (u_{\varepsilon}+e)$ (Lemma \ref{lem-3.5a}).
\begin{lem}\label{lemma-3.6}
Let $l\geqslant1$, let $T = \min \{\widetilde{T} , T_{\max,\varepsilon}\}$ for $\widetilde{T} \in \left(0, +\infty\right)$ and let
\begin{align*}
G_{\varepsilon}(t) =\left\{\begin{array}{lll}
\frac{4b}{(l-3)(l-2)} \int_{\Omega} u_{\varepsilon}^{3-l}+\int_{\Omega} \frac{\left|\nabla v_{\varepsilon}\right|^4}{v_{\varepsilon}^3} & \text { when } & 1 \leqslant l<2  \text {   or   }  l>3, \\
-\frac{4b}{(3-l)(l-2)} \int_{\Omega} u_{\varepsilon}^{3-l}+\int_{\Omega} \frac{\left|\nabla v_{\varepsilon}\right|^4}{v_{\varepsilon}^3} & \text { when } & 2<l<3,\\
4b \int_{\Omega} u_{\varepsilon} \ln u_{\varepsilon}+\int_{\Omega} \frac{\left|\nabla v_{\varepsilon}\right|^4}{v_{\varepsilon}^3} & \text { when } & l=2, \\
-4b \int_{\Omega} \ln u_{\varepsilon}+\int_{\Omega} \frac{\left|\nabla v_{\varepsilon}\right|^4}{v_{\varepsilon}^3} & \text { when } & l=3. 
\end{array}\right.
\end{align*}
Then we have
\begin{align}\label{0922-1558}
G_{\varepsilon}^{\prime}(t) + b \int_{\Omega} v_{\varepsilon}\left|\nabla u_{\varepsilon}\right|^2 + \int_{\Omega} u_{\varepsilon} v_{\varepsilon}^{-3}\left|\nabla v_{\varepsilon}\right|^4 
\leqslant &  4b  \int_{\Omega} u_{\varepsilon}^{2} v_{\varepsilon}\left|\nabla v_{\varepsilon}\right|^2  -\frac{4 b}{l-2}  \int_{\Omega} u^{3-l}_{\varepsilon} + \frac{4 b}{l-2}  \int_{\Omega} u^{4-l}_{\varepsilon}   \nonumber\\  
& + c |\Omega| \left\|v_0\right\|_{L^{\infty}(\Omega)} ~~ ~\text { when }  1 \leqslant l<2  \text {   or   } l>3
\end{align}
and
\begin{align}\label{0922-1561}
G_{\varepsilon}^{\prime}(t) + b \int_{\Omega} v_{\varepsilon}\left|\nabla u_{\varepsilon}\right|^2 + \int_{\Omega} u_{\varepsilon} v_{\varepsilon}^{-3}\left|\nabla v_{\varepsilon}\right|^4  
\leqslant &  4b  \int_{\Omega} u_{\varepsilon}^{2} v_{\varepsilon}\left|\nabla v_{\varepsilon}\right|^2 + \frac{4 b}{l-2}  \int_{\Omega} u^{3-l}_{\varepsilon}  \nonumber\\
& + c |\Omega| \left\|v_0\right\|_{L^{\infty}(\Omega)} \quad \text { when } 2<l<3
\end{align}
and
\begin{align}\label{0922-1559}
G_{\varepsilon}^{\prime}(t) + b \int_{\Omega} v_{\varepsilon}\left|\nabla u_{\varepsilon}\right|^2 & + \int_{\Omega} u_{\varepsilon} v_{\varepsilon}^{-3}\left|\nabla v_{\varepsilon}\right|^4  +4b \int_{\Omega} u^2_{\varepsilon} 
\leqslant   \frac{4b(e+1)}{e} \int_{\Omega} u_{\varepsilon} + 4b \int_{\Omega} u_{\varepsilon}  \ln u_{\varepsilon} \nonumber\\  
&
+ 4b \int_{\Omega} u_{\varepsilon}^{2} v_{\varepsilon}\left|\nabla v_{\varepsilon}\right|^2 + c |\Omega| \left\|v_0\right\|_{L^{\infty}(\Omega)} \quad \text { when } l=2
\end{align}
as well as
\begin{align}\label{0922-1560}
G_{\varepsilon}^{\prime}(t) + b \int_{\Omega} v_{\varepsilon}\left|\nabla u_{\varepsilon}\right|^2 + \int_{\Omega} u_{\varepsilon} v_{\varepsilon}^{-3}\left|\nabla v_{\varepsilon}\right|^4   
\leqslant &  4b \int_{\Omega} u_{\varepsilon}^{2} v_{\varepsilon}\left|\nabla v_{\varepsilon}\right|^2 + 4b \int_{\Omega} u_{\varepsilon} \nonumber\\
& +  c |\Omega| \left\|v_0\right\|_{L^{\infty}(\Omega)} \quad \text { when } l=3
\end{align}
for all $t \in(0, T)$ and $\varepsilon \in(0,1)$, where $b$ and $c$ are some positive constants independent of $\varepsilon$.
\end{lem}

\begin{proof}
Based on \cite[Lemma 3.4]{2022-DCDSSB-Winkler}, we see that there exists a constant $b>0$ such that
\begin{align}\label{-3.21}
\int_{\Omega} \frac{|\nabla \phi|^6}{\phi^5} \leqslant b \int_{\Omega} \phi^{-1}|\nabla \phi|^2\left|D^2 \ln \phi\right|^2
\end{align}
and
\begin{align}\label{-3.21a}
\int_{\Omega} \phi^{-3}|\nabla \phi|^{2}\left|D^2 \varphi\right|^2 \leqslant b \int_{\Omega} \phi^{-1}|\nabla \phi|^{2}\left|D^2 \ln \phi\right|^2.
\end{align}
In view of \cite[Lemma 2.3]{2022-JDE-Li}, we obtain 
\begin{align}\label{-3.22}
\frac{d}{d t} \int_{\Omega} \frac{\left|\nabla v_{\varepsilon}\right|^4}{v_{\varepsilon}^3} + & 4 \int_{\Omega} v_{\varepsilon}^{-1}\left|\nabla v_{\varepsilon}\right|^2\left|D^2 \ln v_{\varepsilon}\right|^2  +\int_{\Omega} u_{\varepsilon} v_{\varepsilon}^{-3}\left|\nabla v_{\varepsilon}\right|^4  \nonumber\\
\leqslant & -4 \int_{\Omega} v_{\varepsilon}^{-2}\left|\nabla v_{\varepsilon}\right|^2\left(\nabla u_{\varepsilon} \cdot \nabla v_{\varepsilon}\right)
+ 2 \int_{\partial \Omega} v_{\varepsilon}^{-3}\left|\nabla v_{\varepsilon}\right|^2 \frac{\partial \left|\nabla v_{\varepsilon}\right|^2}{\partial \nu}.
\end{align}
for all $t \in(0, T)$ and $\varepsilon \in(0,1)$. Invoking Young's inequality and \eqref{-3.21}, we deduce that
\begin{align}\label{-3.23}
-4 \int_{\Omega} v_{\varepsilon}^{-2}\left|\nabla v_{\varepsilon}\right|^2\left(\nabla u_{\varepsilon} \cdot \nabla v_{\varepsilon}\right) & \leqslant \frac{2}{b} \int_{\Omega} \frac{\left|\nabla v_{\varepsilon}\right|^6}{v_{\varepsilon}^5}+2 b \int_{\Omega} v_{\varepsilon}\left|\nabla u_{\varepsilon}\right|^2 \nonumber\\
& \leqslant 2 \int_{\Omega} v_{\varepsilon}^{-1}\left|\nabla v_{\varepsilon}\right|^2\left|D^2 \ln v_{\varepsilon}\right|^2 +2 b \int_{\Omega} v_{\varepsilon}\left|\nabla u_{\varepsilon}\right|^2.
\end{align}
In light of \cite[Lemma 3.5]{2022-JDE-Li}, \eqref{-2.9} and \eqref{-3.21a}, we get
\begin{align}\label{-3.23a}
2 \int_{\partial \Omega} v_{\varepsilon}^{-3}|\nabla v_{\varepsilon}|^{2} \cdot \frac{\partial|\nabla v_{\varepsilon}|^2}{\partial \nu} 
\leqslant & \frac{1}{b} \int_{\Omega} v_{\varepsilon}^{-3}|\nabla v_{\varepsilon}|^{2}\left|D^2 v_{\varepsilon}\right|^2+\frac{1}{b} \int_{\Omega} \frac{\left|\nabla v_{\varepsilon}\right|^6}{v_{\varepsilon}^5}+c \int_{\Omega} v_{\varepsilon}\nonumber\\
\leqslant &  2 \int_{\Omega} v_{\varepsilon}^{-1}\left|\nabla v_{\varepsilon}\right|^2\left|D^2 \ln v_{\varepsilon}\right|^2 + c |\Omega| \left\|v_0\right\|_{L^{\infty}(\Omega)}. 
\end{align}
Gathering \eqref{-3.22}-\eqref{-3.23a}, we see that
\begin{align}\label{-3.22aa}
\frac{d}{d t} \int_{\Omega} \frac{\left|\nabla v_{\varepsilon}\right|^4}{v_{\varepsilon}^3} +\int_{\Omega} u_{\varepsilon} v_{\varepsilon}^{-3}\left|\nabla v_{\varepsilon}\right|^4  
\leqslant  2 b \int_{\Omega} v_{\varepsilon}\left|\nabla u_{\varepsilon}\right|^2+ c |\Omega| \left\|v_0\right\|_{L^{\infty}(\Omega)} 
\end{align}
for all $t \in(0, T)$ and $\varepsilon \in(0,1)$. 

The test functions in the proof are different according to the value of $l$. 

Case I: $l\neq2$ and $l\neq3$. Multiplying the first equation in \eqref{sys-regul} by $u^{2-l}_{\varepsilon}$, we infer that
\begin{align}\label{0921-1843a}
\dt \int_{\Omega} u^{3-l}_{\varepsilon} 
= & (3-l)\int_{\Omega} u^{2-l}_{\varepsilon} \cdot \left\{\nabla \cdot \left(u^{l-1}_{\varepsilon} v_{\varepsilon} \nabla u_{\varepsilon }-u^{l}_{\varepsilon} v_{\varepsilon} \nabla v_{\varepsilon }\right)+ u_{\varepsilon} - u_{\varepsilon}^2 \right\} \nonumber\\
= &  (3-l)(l-2)\int_{\Omega} v_{\varepsilon}|\nabla u_{\varepsilon }|^2-(3-l)(l-2) \int_{\Omega} u_{\varepsilon} v_{\varepsilon} \nabla u_{\varepsilon }\cdot \nabla v_{\varepsilon } \nonumber\\
& +(3-l)\int_{\Omega}u^{3-l}_{\varepsilon} - (3-l)\int_{\Omega}u^{4-l}_{\varepsilon}\quad \text { for all } t \in(0, T) \text { and } \varepsilon \in(0,1).
\end{align}
 
In this case, we need to consider two subcases. 

Subcase 1: $1 \leqslant l<2$ or $l>3$. From \eqref{0921-1843a} and Young’s inequality, we obtain
\begin{small}
\begin{align*}
\dt \int_{\Omega} u^{3-l}_{\varepsilon} + (l-3)(l-2)\int_{\Omega} v_{\varepsilon}|\nabla u_{\varepsilon }|^2 
= &  (l-3)(l-2) \int_{\Omega} u_{\varepsilon} v_{\varepsilon} \nabla u_{\varepsilon }\cdot \nabla v_{\varepsilon }+(3-l)\int_{\Omega}u^{3-l}_{\varepsilon} \nonumber\\
&  - (3-l)\int_{\Omega}u^{4-l}_{\varepsilon} \nonumber\\
\leqslant  & \frac{(l-3)(l-2)}{4}\int_{\Omega} v_{\varepsilon}|\nabla u_{\varepsilon }|^2  + (l-3)(l-2) \int_{\Omega} u^{2}_{\varepsilon} v_{\varepsilon}\left|\nabla v_{\varepsilon}\right|^2 \nonumber\\
& + (3-l) \int_{\Omega}u^{3-l}_{\varepsilon}  - (3-l)\int_{\Omega}u^{4-l}_{\varepsilon},
\end{align*} 
\end{small}
so that
\begin{align}\label{0921-1104}
\frac{4b}{(l-3)(l-2)}  \dt \int_{\Omega} u^{3-l}_{\varepsilon} +3 b \int_{\Omega} v_{\varepsilon}\left|\nabla u_{\varepsilon}\right|^2  \leqslant & 4b  \int_{\Omega} u_{\varepsilon}^{2} v_{\varepsilon}\left|\nabla v_{\varepsilon}\right|^2  -\frac{4 b}{l-2}  \int_{\Omega} u^{3-l}_{\varepsilon} \nonumber\\
& + \frac{4 b}{l-2}  \int_{\Omega} u^{4-l}_{\varepsilon}. 
\end{align} 
for all $t \in(0, T)$ and $\varepsilon \in(0,1)$. We set
\begin{align*}
G_{\varepsilon}(t)= \frac{4b}{(l-3)(l-2)}\int_{\Omega} u^{3-l}_{\varepsilon} +\int_{\Omega} \frac{\left|\nabla v_{\varepsilon}\right|^4}{v_{\varepsilon}^3}\quad \text { for all } t \in(0, T) \text { and } \varepsilon \in(0,1).
\end{align*}
Combining \eqref{-3.22aa} with \eqref{0921-1104}, we conclude that
\begin{align*}
G_{\varepsilon}^{\prime}(t) + b \int_{\Omega} v_{\varepsilon}\left|\nabla u_{\varepsilon}\right|^2 + \int_{\Omega} u_{\varepsilon} v_{\varepsilon}^{-3}\left|\nabla v_{\varepsilon}\right|^4 
\leqslant & 4b  \int_{\Omega} u_{\varepsilon}^{2} v_{\varepsilon}\left|\nabla v_{\varepsilon}\right|^2  -\frac{4 b}{l-2}  \int_{\Omega} u^{3-l}_{\varepsilon} + \frac{4 b}{l-2}  \int_{\Omega} u^{4-l}_{\varepsilon} \nonumber\\
& + c |\Omega| \left\|v_0\right\|_{L^{\infty}(\Omega)}
\end{align*}
for all $t \in(0, T)$ and $\varepsilon \in(0,1)$, which implies \eqref{0922-1558}.

Subcase 2: $2 < l<3$. From \eqref{0921-1843a} and Young’s inequality, we see that
\begin{small}
\begin{align*}
- \dt \int_{\Omega} u^{3-l}_{\varepsilon} + (3-l)(l-2)\int_{\Omega} v_{\varepsilon}|\nabla u_{\varepsilon }|^2 
= &  (3-l)(l-2) \int_{\Omega} u_{\varepsilon} v_{\varepsilon} \nabla u_{\varepsilon }\cdot \nabla v_{\varepsilon } \nonumber\\
& -(3-l)\int_{\Omega}u^{3-l}_{\varepsilon} + (3-l)\int_{\Omega}u^{4-l}_{\varepsilon} \nonumber\\
\leqslant  & \frac{(3-l)(l-2)}{4}\int_{\Omega} v_{\varepsilon}|\nabla u_{\varepsilon }|^2  + (3-l)(l-2) \int_{\Omega} u^{2}_{\varepsilon} v_{\varepsilon}\left|\nabla v_{\varepsilon}\right|^2 \nonumber\\
& - (3-l) \int_{\Omega}u^{3-l}_{\varepsilon}  + (3-l)\int_{\Omega}u^{4-l}_{\varepsilon},
\end{align*} 
\end{small}
so that
\begin{align}\label{0921-2120}
-\frac{4b}{(3-l)(l-2)} \dt \int_{\Omega} u^{3-l}_{\varepsilon} + 3b \int_{\Omega} v_{\varepsilon}|\nabla u_{\varepsilon }|^2 
\leqslant &  4 b \int_{\Omega} u^{2}_{\varepsilon} v_{\varepsilon}\left|\nabla v_{\varepsilon}\right|^2 + \frac{4 b}{l-2}   \int_{\Omega}u^{4-l}_{\varepsilon}
\end{align}
for all $t \in(0, T)$ and $\varepsilon \in(0,1)$. We set
\begin{align*}
G_{\varepsilon}(t)=- \frac{4b}{(3-l)(l-2)}\int_{\Omega} u^{3-l}_{\varepsilon} +\int_{\Omega} \frac{\left|\nabla v_{\varepsilon}\right|^4}{v_{\varepsilon}^3}\quad \text { for all } t \in(0, T) \text { and } \varepsilon \in(0,1).
\end{align*}
Combining \eqref{-3.22aa} with \eqref{0921-2120}, we infer that
\begin{align*}
G_{\varepsilon}^{\prime}(t) + b \int_{\Omega} v_{\varepsilon}\left|\nabla u_{\varepsilon}\right|^2 + \int_{\Omega} u_{\varepsilon} v_{\varepsilon}^{-3}\left|\nabla v_{\varepsilon}\right|^4 
\leqslant &  4b  \int_{\Omega} u_{\varepsilon}^{2} v_{\varepsilon}\left|\nabla v_{\varepsilon}\right|^2 + \frac{4 b}{l-2}  \int_{\Omega} u^{3-l}_{\varepsilon}\\
& + c |\Omega| \left\|v_0\right\|_{L^{\infty}(\Omega)},
\end{align*}
for all $t \in(0, T)$ and $\varepsilon \in(0,1)$, which implies \eqref{0922-1561}.

Case II: $l=2$. Multiplying the first equation in \eqref{sys-regul} by $1+\ln u_{\varepsilon}$, we use Cauchy-Schwarz inequality and $\xi \ln \xi +\frac{1}{e}\geqslant 0$ for all $\xi>0$ to obtain
\begin{align*}
\dt \int_{\Omega} u_{\varepsilon} \ln u_{\varepsilon}+\int_{\Omega} v_{\varepsilon}\left|\nabla u_{\varepsilon}\right|^2 
= &  \int_{\Omega} u_{\varepsilon} v_{\varepsilon} \nabla u_{\varepsilon} \cdot \nabla v_{\varepsilon}+ \int_{\Omega} u_{\varepsilon}   - \int_{\Omega} u^2_{\varepsilon}   + \int_{\Omega} u_{\varepsilon}  \ln u_{\varepsilon} - \int_{\Omega} u^2_{\varepsilon}  \ln u_{\varepsilon}\nonumber\\
\leqslant &   \frac{1}{4} \int_{\Omega} v_{\varepsilon}\left|\nabla u_{\varepsilon}\right|^2+  \int_{\Omega} u_{\varepsilon}^{2} v_{\varepsilon}\left|\nabla v_{\varepsilon}\right|^2 + \int_{\Omega} u_{\varepsilon}   - \int_{\Omega} u^2_{\varepsilon}   + \int_{\Omega} u_{\varepsilon}  \ln u_{\varepsilon} \nonumber\\
& + \frac{1}{e} \int_{\Omega} u_{\varepsilon} \nonumber\\
= &   \frac{1}{4} \int_{\Omega} v_{\varepsilon}\left|\nabla u_{\varepsilon}\right|^2+  \int_{\Omega} u_{\varepsilon}^{2} v_{\varepsilon}\left|\nabla v_{\varepsilon}\right|^2 + \frac{e+1}{e} \int_{\Omega} u_{\varepsilon}  - \int_{\Omega} u^2_{\varepsilon}  \nonumber\\
&  + \int_{\Omega} u_{\varepsilon}  \ln u_{\varepsilon}, 
\end{align*}
so that
\begin{align}\label{0921-1256}
4b \dt \int_{\Omega} u_{\varepsilon} \ln u_{\varepsilon}+3b \int_{\Omega} v_{\varepsilon}\left|\nabla u_{\varepsilon}\right|^2 +4b \int_{\Omega} u^2_{\varepsilon}   
\leqslant &   \frac{4b(e+1)}{e} \int_{\Omega} u_{\varepsilon}  + 4b \int_{\Omega} u_{\varepsilon}  \ln u_{\varepsilon} \nonumber\\
&
+ 4b \int_{\Omega} u_{\varepsilon}^{2} v_{\varepsilon}\left|\nabla v_{\varepsilon}\right|^2 
\end{align}
for all $t \in(0, T)$ and $\varepsilon \in(0,1)$. We set
\begin{align*}
G_{\varepsilon}(t)= 4b \int_{\Omega} u_{\varepsilon} \ln u_{\varepsilon} +\int_{\Omega} \frac{\left|\nabla v_{\varepsilon}\right|^4}{v_{\varepsilon}^3}\quad \text { for all } t \in(0, T) \text { and } \varepsilon \in(0,1).
\end{align*}
Combining \eqref{-3.22aa} with \eqref{0921-1256}, we conclude that
\begin{align*}
G_{\varepsilon}^{\prime}(t) + b \int_{\Omega} v_{\varepsilon}\left|\nabla u_{\varepsilon}\right|^2 + \int_{\Omega} u_{\varepsilon} v_{\varepsilon}^{-3}\left|\nabla v_{\varepsilon}\right|^4 +4b \int_{\Omega} u^2_{\varepsilon}   
\leqslant & \frac{4b(e+1)}{e} \int_{\Omega} u_{\varepsilon} + 4b \int_{\Omega} u_{\varepsilon}  \ln u_{\varepsilon} \nonumber\\
& + 4b \int_{\Omega} u_{\varepsilon}^{2} v_{\varepsilon}\left|\nabla v_{\varepsilon}\right|^2 + c |\Omega| \left\|v_0\right\|_{L^{\infty}(\Omega)},
\end{align*}
for all $t \in(0, T)$ and $\varepsilon \in(0,1)$, which implies \eqref{0922-1559}.

Case III: $l=3$. Multiplying the first equation in \eqref{sys-regul} by $-u^{-l}_{\varepsilon}$, we use Cauchy-Schwarz inequality to show that
\begin{align*}
- \dt \int_{\Omega}\ln u_{\varepsilon} 
= & - \int_{\Omega} u^{-1}_{\varepsilon} \cdot \left\{\nabla \cdot \left(u^{2}_{\varepsilon} v_{\varepsilon} \nabla u_{\varepsilon }-u^{3}_{\varepsilon} v_{\varepsilon} \nabla v_{\varepsilon }\right)+ u_{\varepsilon} - u_{\varepsilon}^2 \right\} \nonumber\\
= &  -\int_{\Omega} v_{\varepsilon}|\nabla u_{\varepsilon }|^2 + \int_{\Omega} u_{\varepsilon} v_{\varepsilon} \nabla u_{\varepsilon }\cdot \nabla v_{\varepsilon } 
-|\Omega| + \int_{\Omega}u_{\varepsilon}\nonumber\\
\leqslant &  - \frac{3}{4} \int_{\Omega} v_{\varepsilon}\left|\nabla u_{\varepsilon}\right|^2+  \int_{\Omega} u_{\varepsilon}^{2} v_{\varepsilon}\left|\nabla v_{\varepsilon}\right|^2 + \int_{\Omega} u_{\varepsilon}, 
\end{align*} 
so that
\begin{align}\label{0921-2309}
- 4b \dt \int_{\Omega}\ln u_{\varepsilon} + 3b \int_{\Omega} v_{\varepsilon}\left|\nabla u_{\varepsilon}\right|^2
\leqslant  4b \int_{\Omega} u_{\varepsilon}^{2} v_{\varepsilon}\left|\nabla v_{\varepsilon}\right|^2 + 4b \int_{\Omega} u_{\varepsilon}
\end{align}
for all $t \in(0, T)$ and $\varepsilon \in(0,1)$. We set
\begin{align*}
G_{\varepsilon}(t)= -4b \int_{\Omega} \ln u_{\varepsilon} +\int_{\Omega} \frac{\left|\nabla v_{\varepsilon}\right|^4}{v_{\varepsilon}^3}\quad \text { for all } t \in(0, T) \text { and } \varepsilon \in(0,1).
\end{align*}
Combining \eqref{-3.22aa} with \eqref{0921-2309}, we conclude that
\begin{align*}
G_{\varepsilon}^{\prime}(t) + b \int_{\Omega} v_{\varepsilon}\left|\nabla u_{\varepsilon}\right|^2 + \int_{\Omega} u_{\varepsilon} v_{\varepsilon}^{-3}\left|\nabla v_{\varepsilon}\right|^4  
\leqslant  4b \int_{\Omega} u_{\varepsilon}^{2} v_{\varepsilon}\left|\nabla v_{\varepsilon}\right|^2 + 4b \int_{\Omega} u_{\varepsilon} +  c |\Omega| \left\|v_0\right\|_{L^{\infty}(\Omega)}
\end{align*}
for all $t \in(0, T)$ and $\varepsilon \in(0,1)$, which implies \eqref{0922-1560}.
\end{proof}
We are now in a position to derive the estimate for $\frac{|\nabla v_{\varepsilon }|^4}{v_{\varepsilon}^3}$ in $L^\infty((0, T), L^1(\Omega))$ independent of $\varepsilon$ applying Lemma \ref{lemma-3.5} and  the above energy-like functional differential inequality.
\begin{lem}\label{lemma-3.8}
Let $l\geqslant1$ and $T = \min \{\widetilde{T} , T_{\max,\varepsilon}\}$ for $\widetilde{T} \in \left(0, +\infty\right)$, and assume that \eqref{assIniVal} holds. Then there exists $C(T)>0$ such that
\begin{align}\label{-3.29}
\int_{\Omega} \frac{|\nabla v_{\varepsilon}|^4}{v_{\varepsilon}^3} \leqslant C(T) \quad \text { for all } t \in(0, T) \text { and } \varepsilon \in(0,1),
\end{align}
where $C(T)$ is a positive constant depending on $ \int_{\Omega} \ln u_0$, $\int_{\Omega} |\nabla v_{0 }|^2$, $\int_{\Omega} \frac{|\nabla v_{0 }|^2}{v_0}$ and $\int_{\Omega} u_{0}^{3-l}$ $\int_{\Omega} u_0$, but independent of $\varepsilon$. 
 \end{lem}
\begin{proof}
An application of \eqref{eq-6.4} with $p=1$ and $\eta =\frac{1}{8}$ provides $c_1>0$ such that
\begin{align}\label{0922-1640}
 4b   \int_\Omega u_{\varepsilon}^{2} v_{\varepsilon} |\nabla v_{\varepsilon}|^2
\leqslant & \frac{b}{2} \int_\Omega  v_{\varepsilon}|^2\nabla u_{\varepsilon}|^2 + \left\{c_1b \left\|v_{\varepsilon} \right\|_{L^{\infty}(\Omega)}+8 c_1 b    \left\|v_{\varepsilon} \right\|^4_{L^{\infty}(\Omega)}\right\} \cdot \int_{\Omega} u_{\varepsilon}^{2} \cdot \int_\Omega \frac{|\nabla v_{\varepsilon}|^4}{v_{\varepsilon}^3} \nonumber\\
&+c_1 b \left\|v_{\varepsilon} \right\|^2_{L^{\infty}(\Omega)} \cdot\left\{\int_\Omega u_{\varepsilon}\right\}^{3} \cdot \int_\Omega \frac{|\nabla v_{\varepsilon}|^4}{v_{\varepsilon}^3}+c_1 b  \left\|v_{\varepsilon} \right\|^2_{L^{\infty}(\Omega)}\cdot \int_{\Omega} u_{\varepsilon} v_{\varepsilon}.
\end{align}

The proofs of this lemma are different according to the value of $l$.

Case I: $ 1 \leqslant l < 2$. Substituting \eqref{0922-1640} into \eqref{0922-1558}, thanks to \eqref{-2.9}, \eqref{-3.4} and Young’s inequality, for all $t \in(0, T)$ and $\varepsilon \in(0,1)$, we obtain 
\begin{align}\label{-3.33}
G_{\varepsilon}^{\prime}(t) + \frac{b}{2} \int_{\Omega} v_{\varepsilon}\left|\nabla u_{\varepsilon}\right|^2 + \int_{\Omega} u_{\varepsilon} v_{\varepsilon}^{-3}\left|\nabla v_{\varepsilon}\right|^4 
\leqslant & B \int_{\Omega} u_{\varepsilon}^{2}   \cdot \int_\Omega \frac{|\nabla v_{\varepsilon}|^4}{v_{\varepsilon}^3} +B \int_\Omega \frac{|\nabla v_{\varepsilon}|^4}{v_{\varepsilon}^3}+B \int_{\Omega} u_{\varepsilon} v_{\varepsilon} \nonumber\\
& + \frac{4 b}{2-l}  \int_{\Omega} u^{3-l}_{\varepsilon}    + c |\Omega| \left\|v_0\right\|_{L^{\infty}(\Omega)}\nonumber\\
\leqslant & B \int_{\Omega} u_{\varepsilon}^{2}   \cdot \left(\frac{4b}{(l-3)(l-2)} \int_{\Omega} u_{\varepsilon}^{3-l}+\int_\Omega \frac{|\nabla v_{\varepsilon}|^4}{v_{\varepsilon}^3} \right)\nonumber\\
& +B \int_\Omega \frac{|\nabla v_{\varepsilon}|^4}{v_{\varepsilon}^3}+B \int_{\Omega} u_{\varepsilon} v_{\varepsilon} + \frac{4 b}{2-l}  \int_{\Omega} u^{2}_{\varepsilon} \nonumber\\
& + c |\Omega| \left\|v_0\right\|_{L^{\infty}(\Omega)}+\frac{4 b |\Omega|}{2-l}\nonumber\\
= &  B  \int_{\Omega} u_{\varepsilon}^{2}   \cdot G_{\varepsilon} (t)+B \int_\Omega \frac{|\nabla v_{\varepsilon}|^4}{v_{\varepsilon}^3}+B \int_{\Omega} u_{\varepsilon} v_{\varepsilon} \nonumber\\
&+ \frac{4 b}{2-l}  \int_{\Omega} u^{2}_{\varepsilon} + c |\Omega| \left\|v_0\right\|_{L^{\infty}(\Omega)}+\frac{4 b |\Omega|}{2-l},
\end{align}
where $B=\max \left\{c_1 b\left\|v_{0} \right\|_{L^{\infty}(\Omega)}+8 c_1 b \left\|v_{0} \right\|^4_{L^{\infty}(\Omega)}, c_1b m_*^3 \left\|v_{0} \right\|^2_{L^{\infty}(\Omega)}, c_1b \left\|v_{0} \right\|^2_{L^{\infty}(\Omega)} \right\}$. Set
\begin{align*}
Z_{1 \varepsilon}(t)= B\int_{\Omega} u_{\varepsilon}^{2}  \quad \text { for all } t \in(0, T) \text { and } \varepsilon \in(0,1) 
\end{align*}
and
\begin{align*}
M_{1\varepsilon}(t) = & B \int_\Omega \frac{|\nabla v_{\varepsilon}|^4}{v_{\varepsilon}^3} + B \int_{\Omega} u_{\varepsilon} v_{\varepsilon} + \frac{4 b}{2-l}  \int_{\Omega} u^{2}_{\varepsilon} + c |\Omega| \left\|v_0\right\|_{L^{\infty}(\Omega)}+\frac{4 b |\Omega|}{2-l}
\end{align*}
for all $t \in(0, T)$ and $\varepsilon \in(0,1)$. Then we rewrite \eqref{-3.33} as follows
\begin{align*}
G_{\varepsilon}^{\prime}(t) \leqslant  Z_{1\varepsilon}(t) G_{\varepsilon}(t)+M_{1\varepsilon}(t) \quad \text { for all } t \in(0, T) \text { and } \varepsilon \in(0,1). 
\end{align*}
Integrating this differential inequality gives
\begin{align}\label{-3.34}
G_{\varepsilon}(t)  & \leqslant G_{\varepsilon}(0) e^{\int_0^t Z_{1 \varepsilon}(s) \d s}+\int_0^t e^{\int_s^t Z_{1 \varepsilon}(\sigma) \d \sigma} M_{1 \varepsilon}(s) \d s 
\end{align}
for all $t \in(0, T)$ and $\varepsilon \in(0,1)$. Since 
\begin{align*}
\int_s^t Z_{1 \varepsilon}(\sigma) \d \sigma \leqslant  c_2(T) \quad \text { for all } t \in\left(0, T\right),~ s \in[0, t) \text {  and   } \varepsilon \in(0,1)
\end{align*}
by Lemma \ref{lem-1st-est1}, and 
\begin{align*}
\int_0^t M_{1 \varepsilon}(s) \d s \leqslant  c_3(T) \quad \text { for all } t \in\left(0, T\right) \text { and } \varepsilon \in(0,1)
\end{align*}
due to \eqref{-2.10}, \eqref{-3.5} and \eqref{-3.10}, from \eqref{-3.34} we conclude that 
\begin{align}\label{0908-1147}
G_{\varepsilon}(t) & \leqslant  \left\{\frac{4b}{(l-3)(l-2)} \int_{\Omega} (u_{0}+1)^{3-l}+\int_{\Omega} \frac{\left|\nabla v_{0}\right|^4}{v_{0}^3}\right\} \cdot e^{c_2(T)}
+c_3(T) e^{c_2(T)}
\end{align}
for all $t \in(0, T)$ and $\varepsilon \in(0,1)$, which establishes \eqref{-3.29}. Similarly, we can also obtain \eqref{-3.29} for the cases where $2 < l < 3$ or $l>3$.

Case II: $l = 2$. Plugging \eqref{0922-1640} into \eqref{0922-1559} and using $\xi \ln \xi +\frac{1}{e}\geqslant 0$ and $\ln \xi - \xi \leqslant 0$ for all $\xi>0$, thanks to \eqref{-2.9}, \eqref{-3.4} and Young’s inequality, for all $t \in(0, T)$ and $\varepsilon \in(0,1)$, we obtain 
\begin{align}\label{-3.33aa}
G_{\varepsilon}^{\prime}(t) + \frac{b}{2} \int_{\Omega} v_{\varepsilon}\left|\nabla u_{\varepsilon}\right|^2 + \int_{\Omega} u_{\varepsilon} v_{\varepsilon}^{-3}\left|\nabla v_{\varepsilon}\right|^4 \leqslant & B \int_{\Omega} u_{\varepsilon}^{2}   \cdot \int_\Omega \frac{|\nabla v_{\varepsilon}|^4}{v_{\varepsilon}^3} +B \int_\Omega \frac{|\nabla v_{\varepsilon}|^4}{v_{\varepsilon}^3}+B \int_{\Omega} u_{\varepsilon} v_{\varepsilon} \nonumber\\
& +\frac{4b(e+1)}{e} m_*   + c |\Omega| \left\|v_0\right\|_{L^{\infty}(\Omega)} \nonumber\\
\leqslant & B \int_{\Omega} u_{\varepsilon}^{2}   \cdot \left(4b \int_{\Omega} u_{\varepsilon} \ln u_{\varepsilon} +\int_\Omega \frac{|\nabla v_{\varepsilon}|^4}{v_{\varepsilon}^3}+\frac{4b|\Omega|}{e} \right)\nonumber\\
& +B \int_\Omega \frac{|\nabla v_{\varepsilon}|^4}{v_{\varepsilon}^3}+B \int_{\Omega} u_{\varepsilon} v_{\varepsilon}   +\frac{4b(e+1)}{e} m_*   + \nonumber\\
& c |\Omega| \left\|v_0\right\|_{L^{\infty}(\Omega)}\nonumber\\
= &  B  \int_{\Omega} u_{\varepsilon}^{2}   \cdot G_{\varepsilon} (t)+B \int_\Omega \frac{|\nabla v_{\varepsilon}|^4}{v_{\varepsilon}^3}+B \int_{\Omega} u_{\varepsilon} v_{\varepsilon} \nonumber\\
& +\frac{4b |\Omega|B}{e} \int_{\Omega} u^2_{\varepsilon} +\frac{4b(e+1)}{e} m_*   \nonumber\\
& + c |\Omega| \left\|v_0\right\|_{L^{\infty}(\Omega)}
\end{align}
We set
\begin{align*}
Z_{2 \varepsilon}(t)= B\int_{\Omega} u_{\varepsilon}^{2} \quad \text { for all } t \in\left(0, T\right) \text { and } \varepsilon \in(0,1)
\end{align*}
and
\begin{align*}
M_{2 \varepsilon}(t) = B \int_\Omega \frac{|\nabla v_{\varepsilon}|^4}{v_{\varepsilon}^3} + B\int_{\Omega} u_{\varepsilon} v_{\varepsilon}  +4b\left(1+ \frac{|\Omega|B}{e}\right) \int_{\Omega} u^2_{\varepsilon} +\frac{4b(e+1)}{e} m_* + c |\Omega| \left\|v_0\right\|_{L^{\infty}(\Omega)} 
\end{align*}
for all $t \in(0, T)$ and $\varepsilon \in(0,1)$. Then we rewrite \eqref{-3.33aa} as follows
\begin{align*}
G_{\varepsilon}^{\prime}(t) \leqslant  Z_{2 \varepsilon}(t) G_{\varepsilon}(t)+M_{2 \varepsilon}(t) \quad \text { for all } t \in(0, T) \text { and } \varepsilon \in(0,1). 
\end{align*}
Integrating this differential inequality gives
\begin{align}\label{-3.34}
G_{\varepsilon}(t)  & \leqslant G_{\varepsilon}(0) e^{\int_0^t Z_{2 \varepsilon}(s) \d s}+\int_0^t e^{\int_s^t Z_{2 \varepsilon}(\sigma) \d \sigma} M_{2 \varepsilon}(s) \d s
\end{align}
for all $t \in(0, T)$ and $\varepsilon \in(0,1)$. Since 
\begin{align*}
\int_s^t Z_{2 \varepsilon}(\sigma) \d \sigma \leqslant  c_4(T) \quad \text { for all } t \in\left(0, T\right),~ s \in[0, t) \text {  and   } \varepsilon \in(0,1)
\end{align*}
by Lemma \ref{lem-1st-est1}, and 
\begin{align*}
\int_0^t M_{2 \varepsilon}(s) \d s \leqslant  c_5(T) \quad \text { for all } t \in\left(0, T\right) \text { and } \varepsilon \in(0,1)
\end{align*}
due to \eqref{-2.10}, \eqref{-3.5} and \eqref{-3.10}, from \eqref{-3.34} we conclude that 
\begin{align}\label{0908-1147a}
G_{\varepsilon}(t) & \leqslant  \left\{4b \int_{\Omega} (u_{0}+1)\ln (u_{0}+1)+\int_{\Omega} \frac{\left|\nabla v_{0}\right|^4}{v_{0}^3}\right\} \cdot e^{c_4(T)}
+c_5(T) e^{c_4(T)}
\end{align}
for all $t \in(0, T)$ and $\varepsilon \in(0,1)$, which establishes \eqref{-3.29}. Similarly, we can also obtain \eqref{-3.29} for the case where $l = 3$.   
\end{proof}

In \cite[Lemma 3.1]{2024-AaA-LiWinkler}, Li and Winkler derived the estimate of the temporal-spatial integral of $u^2_{\varepsilon} \ln^2 (u_{\varepsilon}+e)$ for a doubly degenerate reaction-diffusion system with logistic source. Based on Lemma \ref{lem-1st-est2}, using the similar method, we also obtain such estimate for the system \eqref{sys-regul} with logistic source.

\begin{lem}\label{lem-3.5a}
Let $l\geqslant1$ and $T = \min \{\widetilde{T} , T_{\max,\varepsilon}\}$ for $\widetilde{T} \in \left(0, +\infty\right)$, and assume that \eqref{assIniVal} holds. Then there exists $C(T)>0$ such that
\begin{align}\label{0923-2002}
\int_0^T \int_{\Omega} u^2_{\varepsilon} \ln^2 (u_{\varepsilon}+e) \leqslant C(T)   \quad \text { for all } \varepsilon \in(0,1),
\end{align}
where $C(T)$ is a positive constant depending on $ \int_{\Omega} \ln u_0$, $\int_{\Omega} |\nabla v_{0 }|^2$, $\int_{\Omega} \frac{|\nabla v_{0 }|^2}{v_0}$ and $\int_{\Omega} u_{0}^{3-l}$ $\int_{\Omega} u_0$, but independent of $\varepsilon$. 
\end{lem}
\begin{proof}
We set
\begin{align*}
\Psi(\xi)=(\xi+e) \ln ^2(\xi+e)-2(\xi+e) \ln (\xi+e)+2(\xi+e) \quad \text { for all } \xi \geqslant 0.
\end{align*}
We rewrite the first equation in \eqref{sys-regul} as 
\begin{align*}
u_{\varepsilon t} =\nabla \cdot\left\{u^{l-1}_{\varepsilon} v_{\varepsilon} e^{v_{\varepsilon}} \nabla\left(u_{\varepsilon} e^{-v_{\varepsilon}}\right)\right\}+u_{\varepsilon}-u_{\varepsilon}^2 \quad \text { for all } t \in\left(0, T\right) \text { and } \varepsilon \in(0,1).
\end{align*}
Multiplying the above equation by $\Psi^{\prime}\left(u_{\varepsilon} e^{-v_{\varepsilon}}\right)$, we have
\begin{align}\label{0923-1629}
\int_{\Omega} \Psi^{\prime}\left(u_{\varepsilon} e^{-v_{\varepsilon}}\right) u_{\varepsilon t}= & -\int_{\Omega} u^{l-1}_{\varepsilon} v_{\varepsilon} e^{v_{\varepsilon}} \Psi^{\prime \prime}\left(u_{\varepsilon} e^{-v_{\varepsilon}}\right)\left|\nabla\left(u_{\varepsilon} e^{-v_{\varepsilon}}\right)\right|^2 \nonumber\\
& +\int_{\Omega}\left(u_{\varepsilon}-u_{\varepsilon}^2\right) \ln ^2\left(u_{\varepsilon} e^{-v_{\varepsilon}}+e\right) 
\end{align}
for all $t \in(0, T)$ and $\varepsilon \in(0,1)$. Letting 
\begin{align}\label{0923-1626}
c_1 = \max \left\{e, e^{2\left\|v_0\right\|_{L^{\infty}(\Omega)}}\right\}.
\end{align}
Based on \eqref{-2.9} and \eqref{0923-1626}, we infer that
\begin{align}\label{0923-1627}
\int_{\left\{u_{\varepsilon} \leqslant c_1\right\}}  \left(u_{\varepsilon}-u_{\varepsilon}^2\right) \ln ^2\left(u_{\varepsilon} e^{-v_{\varepsilon}}+e\right)  & \leqslant \int_{\left\{u_{\varepsilon}  \leqslant  c_1 \right\}} u_{\varepsilon} \ln ^2\left(u_{\varepsilon} e^{-v_{\varepsilon}}+e\right) \nonumber\\
& \leqslant c_1 \ln ^2(c_1+e)|\Omega|
\end{align}
and
\begin{align}\label{0923-1628}
u_{\varepsilon} e^{-v_{\varepsilon}}+e & =\sqrt{u_{\varepsilon}} \cdot e^{\frac{1}{2} \ln u_{\varepsilon}-v_{\varepsilon}}+e \nonumber\\
& \geqslant \sqrt{u_{\varepsilon}} \cdot e^{\frac{1}{2} \ln c_1-\left\|v_0\right\|_{L^{\infty}(\Omega)}} \nonumber\\
& \geqslant \sqrt{u_{\varepsilon}} \quad \text { in }\left\{u_{\varepsilon}>c_1\right\}.
\end{align}
From \eqref{0923-1626}, we obtain $1 \leqslant \frac{1}{e} u_{\varepsilon}$ in $\left\{u_{\varepsilon}>c_1\right\}$, which together with \eqref{0923-1628} yields
\begin{align}\label{0923-1659}
& \int_{\left\{u_{\varepsilon}>c_1\right\}}\left(u_{\varepsilon}-u_{\varepsilon}^2\right) \ln ^2\left(u_{\varepsilon} e^{-v_{\varepsilon}}+e\right) \nonumber\\
= & -\int_{\left\{u_{\varepsilon}>c_1\right\}} u_{\varepsilon}\left(u_{\varepsilon}-1\right) \ln ^2\left(u_{\varepsilon} e^{-v_{\varepsilon}}+e\right) \nonumber\\
\leqslant &  -\left(1-\frac{1}{e}\right) \int_{\left\{u_{\varepsilon}>c_1\right\}} u_{\varepsilon}^2 \ln ^2 \sqrt{u_{\varepsilon}} \nonumber\\
=& - \frac{1-\frac{1}{e}}{4} \int_{\left\{u_{\varepsilon}>c_1\right\}} u_{\varepsilon}^2 \ln ^2 u_{\varepsilon} \nonumber\\
= & -\frac{1-\frac{1}{e}}{4} \int_{\left\{u_{\varepsilon}>c_1\right\}}\left(u_{\varepsilon}+e\right)^2 \ln ^2\left(u_{\varepsilon}+e\right) \cdot\left(\frac{u_{\varepsilon}}{u_{\varepsilon}+e}\right)^2\left(\frac{\ln u_{\varepsilon}}{\ln \left(u_{\varepsilon}+e\right)}\right)^2 \nonumber\\
\leqslant & -\frac{1-\frac{1}{e}}{4} \int_{\left\{u_{\varepsilon}>c_1\right\}}\left(u_{\varepsilon}+e\right)^2 \ln ^2\left(u_{\varepsilon}+e\right) \cdot\left(\frac{e}{e+e}\right)^2\left(\frac{\ln e}{\ln (e+e)}\right)^2 \nonumber\\
= & - c_2 \int_{\left\{u_{\varepsilon}>c_1\right\}}\left(u_{\varepsilon}+e\right)^2 \ln ^2\left(u_{\varepsilon}+e\right), 
\end{align}
where $c_2=\frac{1-\frac{1}{e}}{16(1+\ln 2)^2}$. Combining \eqref{0923-1628} with \eqref{0923-1659}, we see that
\begin{align}\label{0923-1922}
& c_2 \int_{\Omega}\left(u_{\varepsilon}+e\right)^2 \ln ^2\left(u_{\varepsilon}+e\right) \nonumber\\
= & c_2 \int_{\left\{u_{\varepsilon} \leqslant c_1\right\}}\left(u_{\varepsilon}+e\right)^2 \ln ^2\left(u_{\varepsilon}+e\right)+c_2  \int_{\left\{u_{\varepsilon}>c_1\right\}}\left(u_{\varepsilon}+e\right)^2 \ln ^2\left(u_{\varepsilon}+e\right) \nonumber\\
\leqslant &  c_3+c_2 \int_{\left\{u_{\varepsilon}>c_1\right\}}\left(u_{\varepsilon}+e\right)^2 \ln ^2\left(u_{\varepsilon}+e\right) \nonumber\\
\leqslant &  c_3-\int_{\left\{u_{\varepsilon}>c_1\right\}}\left(u_{\varepsilon}-u_{\varepsilon}^2\right) \ln ^2\left(u_{\varepsilon} e^{-v_{\varepsilon}}+e\right) \nonumber\\
= & c_3-\int_{\Omega}\left(u_{\varepsilon}-u_{\varepsilon}^2\right) \ln ^2\left(u_{\varepsilon} e^{-v_{\varepsilon}}+e\right)+\int_{\left\{u_{\varepsilon} \leqslant c_1\right\}}\left(u_{\varepsilon}-u_{\varepsilon}^2\right) \ln ^2\left(u_{\varepsilon} e^{-v_{\varepsilon}}+e\right) \nonumber\\
\leqslant &  c_3+c_1 \ln ^2(c_1+e)|\Omega|-\int_{\Omega}\left(u_{\varepsilon}-u_{\varepsilon}^2\right) \ln ^2\left(u_{\varepsilon} e^{-v_{\varepsilon}}+e\right),
\end{align}
where $c_3=\left(c_1+\right.$ $e)^2 \ln ^2\left(c_1+e\right) \cdot c_2|\Omega|$. By the left-hand side of \eqref{0923-1629}, we ensure that
\begin{align}\label{0923-1923}
& \int_{\Omega} \Psi^{\prime}\left(u_{\varepsilon} e^{-v_{\varepsilon}}\right) u_{\varepsilon t}\nonumber\\
= & \int_{\Omega} \Psi^{\prime}\left(u_{\varepsilon} e^{-v_{\varepsilon}}\right) \cdot\left(u_{\varepsilon} e^{-v_{\varepsilon}} e^{v_{\varepsilon}}\right)_t \nonumber\\
= & \int_{\Omega} e^{v_{\varepsilon}} \cdot\left(u_{\varepsilon} e^{-v_{\varepsilon}}\right)_t \cdot \Psi^{\prime}\left(u_{\varepsilon} e^{-v_{\varepsilon}}\right)+\int_{\Omega} u_{\varepsilon} e^{-v_{\varepsilon}} \Psi^{\prime}\left(u_{\varepsilon} e^{-v_{\varepsilon}}\right) \partial_t e^{v_{\varepsilon}} \nonumber\\
= & \int_{\Omega} e^{v_{\varepsilon}} \partial_t \Psi\left(u_{\varepsilon} e^{-v_{\varepsilon}}\right)+\int_{\Omega} u_{\varepsilon} e^{-v_{\varepsilon}} \Psi^{\prime}\left(u_{\varepsilon} e^{-v_{\varepsilon}}\right) \partial_t e^{v_{\varepsilon}} \nonumber\\
= & \frac{d}{d t} \int_{\Omega} e^{v_{\varepsilon}} \Psi\left(u_{\varepsilon} e^{-v_{\varepsilon}}\right)-\int_{\Omega} \partial_t e^{v_{\varepsilon}} \cdot \Psi\left(u_{\varepsilon} e^{-v_{\varepsilon}}\right) +\int_{\Omega} u_{\varepsilon} e^{-v_{\varepsilon}} \Psi^{\prime}\left(u_{\varepsilon} e^{-v_{\varepsilon}}\right) \partial_t e^{v_{\varepsilon}} \nonumber\\
= & \frac{d}{d t} \int_{\Omega} e^{v_{\varepsilon}} \Psi \left(u_{\varepsilon} e^{-v_{\varepsilon}}\right)+\int_{\Omega}\left\{u_{\varepsilon} e^{-v_{\varepsilon}} \Psi^{\prime}\left(u_{\varepsilon} e^{-v_{\varepsilon}}\right)-\Psi\left(u_{\varepsilon} e^{-v_{\varepsilon}}\right)\right\} \cdot e^{v_{\varepsilon}} v_{\varepsilon t}. 
\end{align}
Owing to $\ln (\xi+e) \leqslant \xi+e$ and $1 \leqslant \ln (\xi+e)$ for $\xi \geqslant 0$, we have
\begin{align}\label{0923-1858}
\Psi(\xi)-\xi \Psi^{\prime}(\xi) & = e \ln ^2(\xi+e)-2(\xi+e) \ln (\xi+e)+2(\xi+e) \nonumber\\
& \leqslant e (\xi+e) \ln (\xi+e)  \quad \text { for all } \xi \geqslant 0.
\end{align}
In view of \eqref{-2.9}, \eqref{0923-1858} and Young’s inequality, we can deduce that
\begin{align}\label{0923-1924}
\int_{\Omega}\left\{\Psi\left(u_{\varepsilon} e^{-v_{\varepsilon}}\right)-u_{\varepsilon} e^{-v_{\varepsilon}} \Psi^{\prime}\left(u_{\varepsilon} e^{-v_{\varepsilon}}\right)\right\} \cdot e^{v_{\varepsilon}} v_{\varepsilon t} 
\leqslant &  e^{1+\left\|v_0\right\|_{L^{\infty}(\Omega)}} \int_{\Omega}\left(u_{\varepsilon}+e\right) \ln \left(u_{\varepsilon}+e\right)\left|v_{\varepsilon t}\right| \nonumber\\
\leqslant &  \frac{c_2}{2} \int_{\Omega}\left(u_{\varepsilon}+e\right)^2 \ln ^2\left(u_{\varepsilon}+e\right)\nonumber\\
& +\frac{e^{2+2\left\|v_0\right\|_{L_{\infty}(\Omega)}}}{2 c_2} \int_{\Omega} v_{\varepsilon t}^2  
\end{align}
for all $t \in(0, T)$ and $\varepsilon \in(0,1)$. Summing up \eqref{0923-1922}, \eqref{0923-1923} and \eqref{0923-1924}, we have
\begin{align*}
\frac{d}{d t} \int_{\Omega} e^{v_{\varepsilon}} \Psi\left(u_{\varepsilon} e^{-v_{\varepsilon}}\right)
\leqslant & \frac{c_2}{2} \int_{\Omega}\left(u_{\varepsilon}+e\right)^2 \ln ^2\left(u_{\varepsilon}+e\right)+\frac{e^{2+2\left\|v_0\right\|_{L_{\infty}(\Omega)}}}{2 c_2} \int_{\Omega} v_{\varepsilon t}^2 +c_3+c_1 \ln ^2(c_1+e)|\Omega|\\
& -c_2 \int_{\Omega}\left(u_{\varepsilon}+e\right)^2 \ln ^2\left(u_{\varepsilon}+e\right) \quad \text { for all } t \in\left(0, T_{\max , \varepsilon}\right)
\end{align*}
so that
\begin{align*}
\dt \int_{\Omega} e^{v_{\varepsilon}} \Psi\left(u_{\varepsilon} e^{-v_{\varepsilon}}\right)+\frac{c_2}{2} \int_{\Omega}\left(u_{\varepsilon}+e\right)^2 \ln ^2\left(u_{\varepsilon}+e\right)
\leqslant &  \frac{e^{2+2\left\|v_0\right\|_{L_{\infty}(\Omega)}}}{2 c_2} \int_{\Omega} v_{\varepsilon t}^2 +c_1 \ln ^2(c_1+e)|\Omega| +c_3
\end{align*}
for all $t \in(0, T)$ and $\varepsilon \in(0,1)$. Since $\Psi^{\prime}(\xi)\geqslant 0$ for all $\xi \geqslant 0$ implies that $\Psi\left(\left(u_0+\varepsilon\right) e^{-v_0}\right) \leqslant \Psi\left(u_0+1\right)$ and
\begin{align*}
\Psi\left(u_{\varepsilon} e^{-v_{\varepsilon}}\right) \geqslant \Psi(0)=e>0, 
\end{align*}
upon an integration we infer that
\begin{align*}
\frac{c_2}{2} \int_0^T \int_{\Omega}\left(u_{\varepsilon}+e\right)^2 \ln ^2\left(u_{\varepsilon}+e\right) \leqslant & \int_{\Omega} e^{v_0} \Psi\left(u_0+1\right)+\left(c_3+c_1 \ln ^2(c_1+e)|\Omega|\right) T \\
& +\frac{e^{2+2\left\|v_0\right\|_{L_{\infty}(\Omega)}}}{2 c_2} \int_0^T \int_{\Omega} v_{\varepsilon t}^2 \quad \text { for all } \varepsilon \in(0,1),
\end{align*}
thus \eqref{0923-2002} results from \eqref{-3.7}.
\end{proof}

Using Lemma \ref{lem-3.5a} and \cite[Corollary 1.3]{2022-IMRN-Winkler}, we will derive the estimate for $\frac{1}{v_{\varepsilon}} \in L_{\mathrm{loc}}^{\infty}(\bar{\Omega} \times[0, \infty))$ independent of $\varepsilon$.
\begin{lem}\label{lem-14.3} 
Let $l\geqslant1$ and $T = \min \{\widetilde{T} , T_{\max,\varepsilon}\}$ for $\widetilde{T} \in \left(0, +\infty\right)$, and assume that \eqref{assIniVal} holds. Then there exists $C(T)>0$ such that
\begin{align}\label{3.6-1}
\left\|\frac{1}{v_{\varepsilon}(t)}\right\|_{L^{\infty}(\Omega)} \leqslant C(T) \quad \text { for all } t \in(0, T) \text { and } \varepsilon \in(0,1).
\end{align}
where $C(T)$ is a positive constant depending on $ \int_{\Omega} \ln u_0$, $\int_{\Omega} |\nabla v_{0 }|^2$, $\int_{\Omega} \frac{|\nabla v_{0 }|^2}{v_0}$ and $\int_{\Omega} u_{0}^{3-l}$ $\int_{\Omega} u_0$, but independent of $\varepsilon$. 
\end{lem}

\begin{proof}
Let
\begin{align*}
w_{\varepsilon}(x, t)=-\ln \frac{v_{\varepsilon}(x, t)}{\left\|v_0\right\|_{L^{\infty}(\Omega)}},
\end{align*}
then the second equation of system \eqref{sys-regul} with its corresponding boundary condition becomes
\begin{align*}
\begin{cases}w_{\varepsilon t}= \Delta w_{\varepsilon}-|\nabla w_{\varepsilon}|^2+ u_{\varepsilon} & x \in \Omega, t \in\left(0, T\right), \\ \frac{\partial w_{\varepsilon}}{\partial \nu}=0, & x \in \partial \Omega, t \in\left(0, T\right), \\ w_{\varepsilon}(x, 0)=-\ln \frac{v_0(x)}{\left\|v_0\right\|_{L^{\infty}(\Omega)}}, & x \in \Omega .\end{cases}
\end{align*}
In view of
\begin{align*}
w_{\varepsilon t}=\Delta w_{\varepsilon}-|\nabla w_{\varepsilon}|^2+ u_{\varepsilon} \leqslant \Delta w_{\varepsilon}+ u_{\varepsilon}, 
\end{align*}
and the comparison principle, we have
\begin{align}\label{0923-2252}
w_{\varepsilon} \leqslant \overline{w_{\varepsilon}} \quad \text { in } \Omega \times\left(0, T\right),
\end{align}
where $\overline{w_{\varepsilon}} \in \bigcap_{q>2} C^0\left(\left[0, T\right) ; W^{1, q}(\Omega)\right) \cap C^{2,1}\left(\overline{\Omega} \times\left(0, T \right)\right)$ denotes the classical solution of
\begin{align}
\begin{cases}
\overline{w_{\varepsilon}}=\Delta \overline{w_{\varepsilon}}+u_{\varepsilon}(x, t), & x \in \Omega, t \in\left(0, T\right), \\ \frac{\partial w_{\varepsilon}}{\partial \nu}=0, & x \in \partial \Omega, t \in\left(0, T\right), \\ w_{\varepsilon}(x, 0)=\ln \frac{\left\|v_0\right\|_{L^{\infty}(\Omega)}}{v_0(x)}, & x \in \Omega,
\end{cases}
\end{align}
which exists throughout $\Omega \times\left(0, T\right)$ due to \eqref{-2.6} and the H\"{o}lder regularity theory for parabolic equations (cf. \cite{1993-JDE-PorzioVespri}). 
Based on \cite[Corollary 1.3]{2022-IMRN-Winkler} and Lemma \ref{lem-3.5a}, we can find $c(T)>0$ such that 
\begin{align*}
\overline{w_{\varepsilon}}(x, t) \leqslant c(T) \quad \text { for all } x \in \Omega, ~~ t \in(0, T) \text { and } \varepsilon \in(0,1).
\end{align*}
Combining this with \eqref{0923-2252}, we establish \eqref{3.6-1}.
\end{proof}

In view of \eqref{-3.29} and \eqref{3.6-1}, we can derive the following lemma.

%%%%%%%%%%%%%%%%%%%%%%%%%%%%%%%%%%%%%%%%%%%%%%%%%%%%%%%%%%%%%%%%%%
\begin{lem}\label{lemma-3.9x}
Let $l\geqslant1$ and $T = \min \{\widetilde{T}, T_{\max,\varepsilon}\}$ for $\widetilde{T} \in \left(0, +\infty\right)$, and assume that \eqref{assIniVal} holds. Then for all $p >1$, there exists $C(p,T)>0$ such that
\begin{align}\label{-3.36}
\int_{\Omega} u_{\varepsilon}^p(t) \leqslant C(p,T) \quad \text { for all } t \in(0, T) \text { and } \varepsilon \in(0,1)
\end{align}
and
\begin{align}\label{-3.37}
\int_0^{T} \int_{\Omega} u_{\varepsilon}^{p+l-3} \left|\nabla u_{\varepsilon}\right|^2 \leqslant  C(p,T) \quad \text { for all }  \varepsilon \in(0,1),
\end{align}
where $C(p,T)$ is a positive constant depending on $ \int_{\Omega} \ln u_0$, $\int_{\Omega} |\nabla v_{0 }|^2$, $\int_{\Omega} \frac{|\nabla v_{0 }|^2}{v_0}$ and $\int_{\Omega} u_{0}^{3-l}$ $\int_{\Omega} u_0$, but independent of $\varepsilon$. . 
\end{lem}
%%%%%%%%%%%%%%%%%%%%%%%%%%%%%%%%%%%%%%%%%%%%%%%%%%%%%%%%%%%%%%%%%%

\begin{proof}
We employ \eqref{-3.29} and \eqref{3.6-1} to fix $c_1(T) > 0$ and $c_2(T) > 0$ such that
\begin{align*}
\int_{\Omega} \frac{|\nabla v_{\varepsilon}|^4}{v_{\varepsilon}^3} \leqslant c_1(T) \quad \text{ and } \quad \left\|\frac{1}{v_{\varepsilon}(t)}\right\|_{L^{\infty}(\Omega)} \leqslant c_2(T) \quad \text { for all } t \in(0, T) \text { and } \varepsilon \in(0,1).
\end{align*}
Plugging this into \eqref{0704-0024}, by Young's inequality, we obtain 
\begin{align}\label{0924-1213}
\dt \int_{\Omega} u_{\varepsilon}^p 
+\frac{p(p-1)}{(l+p-1)^2} \int_{\Omega}\left|\nabla\left(u_{\varepsilon}^{\frac{l+p-1}{2}} v_{\varepsilon}^{\frac{1}{2}}\right)\right|^2 
\leqslant &  A c^\frac{1}{2}_1 c^\frac{1}{2}_2 \left\{\int_\Omega u_{\varepsilon}^{2(l+p-1)}v_{\varepsilon}^2\right\}^{\frac{1}{2}} \nonumber\\
& +  A c^\frac{1}{2}_1 \left\{\int_\Omega u_{\varepsilon}^{2(l+p-1)}v_{\varepsilon}^2\right\}^{\frac{1}{2}} + p |\Omega|  \nonumber\\
= &  c_3 \left\{\int_\Omega u_{\varepsilon}^{2(l+p-1)}v_{\varepsilon}^2\right\}^{\frac{1}{2}} + p |\Omega|
\end{align}
for all $t \in(0, T)$ and $\varepsilon \in(0,1)$, where $c_3 = A c^\frac{1}{2}_1( c^\frac{1}{2}_2+1)$. By Gagliardo-Nirenberg inequality and Young's inequality, thanks to \eqref{-3.4}, we find $c_4>0$ and $c_5>0$ such that
\begin{align*}
& c_3 \left\|u_{\varepsilon}^{\frac{l+p-1}{2}} v_{\varepsilon}^{\frac{1}{2}}\right\|_{L^{4}(\Omega)}^2 \nonumber\\
\leqslant & c_3 c_4\left\|\nabla\left(u_{\varepsilon}^{\frac{l+p-1}{2}} v_{\varepsilon}^{\frac{1}{2}}\right)\right\|_{L^2(\Omega)}^{2k}\left\|u_{\varepsilon}^{\frac{l+p-1}{2}} v_{\varepsilon}^{\frac{1}{2}}\right\|_{L^{\frac{2}{l+p-1}}(\Omega)}^{2-2k}+c_3 c_4\left\|u_{\varepsilon}^{\frac{l+p-1}{2}} v_{\varepsilon}^{\frac{1}{2}}\right\|_{L^{\frac{2}{l+p-1}}(\Omega)}^{2} \nonumber\\
\leqslant & c_3 c_4 \left\|v_0\right\|^{1-k}_{L^{\infty}(\Omega)} m_*^{(l+p-1)(1-k)} \left\|\nabla\left(u_{\varepsilon}^{\frac{l+p-1}{2}} v_{\varepsilon}^{\frac{1}{2}}\right)\right\|_{L^2(\Omega)}^{2k}+c_3 c_4 \left\|v_0\right\|_{L^{\infty}(\Omega)} m_*^{l+p-1}\nonumber\\
\leqslant & \frac{p(p-1)}{(l+p-1)^2} \left\|\nabla\left(u_{\varepsilon}^{\frac{l+p-1}{2}} v_{\varepsilon}^{\frac{1}{2}}\right)\right\|_{L^2(\Omega)}^{2}+c_5,
\end{align*}
where $k= \frac{2(l+p)-3}{2(l+p)-2} \in (0,1)$. This
together with \eqref{0924-1213} yields
\begin{align}\label{0929-2129}
\dt \int_{\Omega} u_{\varepsilon}^p 
+\frac{p(p-1)}{(l+p-1)^2} \int_{\Omega}\left|\nabla\left(u_{\varepsilon}^{\frac{l+p-1}{2}} v_{\varepsilon}^{\frac{1}{2}}\right)\right|^2 \leqslant c_6 
\end{align}
for all $t \in(0, T)$ and $\varepsilon \in(0,1)$, where $c_6=c_5+p |\Omega|$, which establishes \eqref{-3.36}. By a direct integration in \eqref{0929-2129}, we conclude that
\begin{align*}
\int_0^{T} \int_{\Omega} \left|\nabla\left(u_{\varepsilon}^{\frac{l+p-1}{2}} v_{\varepsilon}^{\frac{1}{2}}\right)\right|^2 \leqslant  C(p,T)\quad \text { for all }  \varepsilon \in(0,1).
\end{align*}
Since \eqref{3.6-1} and
$\int_{\Omega} u_{\varepsilon}^{l+p-3} v_{\varepsilon}|\nabla u_{\varepsilon}|^2 
= \frac{4}{(l+p-1)^2} \int_{\Omega}|v_{\varepsilon}^{\frac{1}{2}} \nabla u_{\varepsilon}^{\frac{l+p-1}{2}}|^2$, we obtain
\begin{align*}
\int_0^{T} \int_{\Omega} u_{\varepsilon}^{p+l-3} \left|\nabla u_{\varepsilon}\right|^2 \leqslant  C(p,T)\quad \text { for all }  \varepsilon \in(0,1).
\end{align*}
This completes the proof of \eqref{-3.37}.
\end{proof}

\section{Global existence of weak solutions}\label{section4}
On the basis of standard heat semigroup estimates we can obtain $L^{\infty}$ bounds for $\nabla v_{\varepsilon}$.

\begin{lem}\label{lemma-4.1}
Let $l\geqslant1$ and $T = \min \{\widetilde{T} , T_{\max,\varepsilon}\}$ for $\widetilde{T} \in \left(0, +\infty\right)$, and assume that \eqref{assIniVal} holds. Then there exists $C(T)>0$ such that
\begin{align}\label{-4.1}
\left\|v_{\varepsilon}(t)\right\|_{W^{1, \infty}(\Omega)} \leqslant C(T)\quad \text { for all } t \in(0, T) \text { and } \varepsilon \in(0,1),
\end{align}
where $C(T)$ is a positive constant independent of $\varepsilon$.
\end{lem}
\begin{proof}
According to the Neumann heat semigroup \cite{2010-JDE-Winkler}, fixing any $p > 1$ we can find $c_1>0$ such that 
\begin{align}\label{-4.2}
&\left\|\nabla v_{\varepsilon}(t)\right\|_{L^{\infty}(\Omega)} \nonumber\\
= &\left\|\nabla e^{t(\Delta-1)} v_0-\int_0^t \nabla e^{(t-s)(\Delta-1)}\left\{u_{\varepsilon}(s) v_{\varepsilon}(s)-v_{\varepsilon}(s)\right\} \d s\right\|_{L^{\infty}(\Omega)} \nonumber\\
\leqslant &  c_1\left\|v_0\right\|_{W^{1, \infty}(\Omega)}\nonumber\\
& +  c_1 \int_0^t\left(1+(t-s)^{-\frac{1}{2}-\frac{1}{p}}\right) e^{-(t-s)}\left\|u_{\varepsilon}(s) v_{\varepsilon}(s)-v_{\varepsilon}(s)\right\|_{L^p(\Omega)} \d s
\end{align}
for all $t \in\left(0, T_{\max, \varepsilon}\right)$ and $\varepsilon \in(0,1)$. In view of \eqref{-2.9} and \eqref{-3.36}, we can find $c_2=c_2(T)>0$ such that
\begin{align*}
\left\|u_{\varepsilon}(s) v_{\varepsilon}(s)-v_{\varepsilon}(s)\right\|_{L^p(\Omega)} 
\leqslant & \left\|u_{\varepsilon}(s)\right\|_{L^p(\Omega)}\left\|v_{\varepsilon}(s)\right\|_{L^{\infty}(\Omega)}+|\Omega|^{\frac{1}{p}}\left\|v_{\varepsilon}(s)\right\|_{L^{\infty}(\Omega)} \\
\leqslant & c_2\left\|v_0\right\|_{L^{\infty}(\Omega)}+|\Omega|^{\frac{1}{p}}\left\|v_0\right\|_{L^{\infty}(\Omega)}.
\end{align*}
This together with \eqref{-4.2} and \eqref{-2.9} yields \eqref{-4.1}.
\end{proof}

We can now proceed to assert the boundedness of $u_{\varepsilon}$.
\begin{lem}\label{lemma-4.4}
Let $l\geqslant1$ and $T = \min \{\widetilde{T} , T_{\max,\varepsilon}\}$ for $\widetilde{T} \in \left(0, +\infty\right)$, and assume that \eqref{assIniVal} holds. Then there exists $C(T)>0$ such that
\begin{align}\label{619-1648}
\left\|u_{\varepsilon}(t)\right\|_{L^{\infty}(\Omega)} \leqslant C(T)\quad \text { for all } t \in(0, T) \text { and } \varepsilon \in(0,1), 
\end{align}
where $C(T)$ is a positive constant independent of $\varepsilon$.
\end{lem}
\begin{proof}
We apply \eqref{3.6-1} and \eqref{-4.1} to find $c_1(T) > 0$ and $c_2(T) > 0$ such that
\begin{align*}
\left(u_{\varepsilon}^{l-1} v_{\varepsilon}\right)(x, t) \geqslant c_1(T) u_{\varepsilon}^{l-1}(x, t) \quad \text { for all } x \in \Omega, t \in(0, T) \text {, and } \varepsilon \in(0,1)
\end{align*}
and
\begin{align}\label{0929-1946}
\left\|\nabla v_{\varepsilon}(\cdot, t)\right\|_{L^{\infty}(\Omega)} \leqslant c_2(T) \quad \text { for all } t \in(0, T) \text { and } \varepsilon \in(0,1).
\end{align}
Since
\begin{align*}
u_{\varepsilon t}=\nabla \cdot\left(u^{l-1}_{\varepsilon} v_{\varepsilon} \nabla u_{\varepsilon}\right)
  - \nabla \cdot\left(u^{l}_{\varepsilon} v_{\varepsilon} \nabla v_{\varepsilon}\right)+ u_{\varepsilon} - u_{\varepsilon}^2, \quad  x \in \Omega, ~~~t>0,
\end{align*}
applying \eqref{-2.9} and \eqref{-3.36} as well as \eqref{0929-1946} implies that for each $q>1$, there exists $c_3(q, T)>0$ such that
\begin{align*}
\sup _{\varepsilon \in(0,1)} \sup _{t \in(0, T)}\left\{\left\|u_{\varepsilon}( t)\right\|_{L^q(\Omega)}+ \left\|u^2_{\varepsilon}( t)\right\|_{L^q(\Omega)}+\left\|\left(u_{\varepsilon}^l v_{\varepsilon} \nabla v_{\varepsilon}\right)(t)\right\|_{L^q(\Omega)}\right\} \leqslant c_3(q, T)
\end{align*}
and finally, a Moser iteration result (cf. \cite{2012-JDE-TaoWinkler}) can imply \eqref{619-1648}.
\end{proof}

According to Lemma \ref{lemma-4.4}, we can demonstrate the global existence of $\left(u_{\varepsilon}, v_{\varepsilon}\right)$ for any $\varepsilon \in(0,1)$.
\begin{lem}\label{lemma-4.5}
Let $l\geqslant1$ and assume that \eqref{assIniVal} holds. Then $T_{\max, \varepsilon}=+\infty$ for all $\varepsilon \in(0,1)$.
\end{lem}
\begin{proof}
This immediately follows from Lemma \ref{lemma-4.4} when combined with \eqref{-2.7}.
\end{proof}

Since the boundedness of $u_{\varepsilon}$ and $v_{\varepsilon}$ asserted in Lemmas  \ref{lem-14.3}, \ref{lemma-4.1} and \ref{lemma-4.4}, the H\"{o}lder estimates of $u_{\varepsilon}$, $v_{\varepsilon}$ and $\nabla v_{\varepsilon}$ can be derived from standard parabolic regularity theory.
\begin{lem}\label{lemma-4.8}
Let $l\geqslant1$ and $T>0$, and assume that \eqref{assIniVal} holds. Then one can find $\theta_1=\theta(T) \in(0,1)$ such that 
\begin{align}\label{-4.12}
\left\|u_{\varepsilon}\right\|_{C^{\theta_1, \frac{\theta_1}{2}}(\overline{\Omega} \times[0, T])} \leqslant C_1(T) \quad \text { for all } \varepsilon \in(0,1)
\end{align}
and
\begin{align}\label{-4.13}
\left\|v_{\varepsilon}\right\|_{C^{\theta_1, \frac{\theta_1}{2}}(\overline{\Omega} \times[0, T])} \leqslant C_1(T) \quad \text { for all } \varepsilon \in(0,1).
\end{align}
where $C_1(T)$ is a positive constant independent of $\varepsilon$.
Moreover, for each $\tau>0$ and all $T>\tau$ one can also fix $\theta_2=\theta_2(\tau, T) \in(0,1)$ such that 
\begin{align}\label{-4.14}
\left\|v_{\varepsilon}\right\|_{C^{2+\theta_2, 1+\frac{\theta_2}{2}}(\overline{\Omega} \times [\tau, T])} \leqslant C_2(\tau, T) \quad \text { for all } \varepsilon \in(0,1)
\end{align}
where $C_2(T)>0$ is a positive constant independent of $\varepsilon$.
\end{lem}
\begin{proof}
It follows from Lemmas \ref{lem-14.3}, \ref{lemma-4.1}, \ref{lemma-4.4} and \eqref{assIniVal} as well as the H\"{o}lder regularity theory for parabolic equations (cf. \cite{1993-JDE-PorzioVespri}) that \eqref{-4.12} and \eqref{-4.13} hold. Combining the standard Schauder estimates (cf. \cite{1968-Ladyzen}) with \eqref{-4.12} and a cut-off argument, we can deduce \eqref{-4.14}.
\end{proof}

In view of the preparations above, we can now utilize a standard extraction procedure to construct a pair of limit functions $(u, v)$, which is proved to be a global bounded weak solution to the system \eqref{SYS:MAIN} as documented in Theorem \ref{thm-1.1}.

\begin{lem}\label{lemma-4.9}
Let $l\geqslant1$ and assume that the initial value $\left(u_0, v_0\right)$ satisfies \eqref{assIniVal}. Then there exist $(\varepsilon_j)_{j \in \mathbb{N}} \subset(0,1)$ as well as functions $u$ and $v$ which satisfy \eqref{solu:property} with $u > 0$ a.e in $\Omega \times(0, \infty)$ and $v>0$ in $\overline{\Omega} \times[0, \infty)$, such that
\begin{flalign}
& u_{\varepsilon} \rightarrow u  \quad \text { in } C_{\mathrm{loc}}^0(\overline{\Omega} \times(0, \infty)), \label{-4.15}\\
& v_{\varepsilon} \rightarrow v  \quad \text { in } C_{\mathrm{loc}}^0(\overline{\Omega} \times[0, \infty)) \text { and in } C_{\mathrm{loc}}^{2,1}(\overline{\Omega} \times(0, \infty)),\label{-4.16}\\
& \nabla v_{\varepsilon} \stackrel{*}{\rightharpoonup} \nabla v \quad \text { in } L^{\infty}(\Omega \times(0, \infty)),\label{-4.17}
\end{flalign}
as $\varepsilon=\varepsilon_j \searrow 0$, and that $(u, v)$ is a global weak solution of (\ref{SYS:MAIN}) according to Definition \ref{def-weak-sol}. 
\end{lem}
\begin{proof}
The existence of $\left\{\varepsilon_j\right\}_{j \in \mathbb{N}}$ and nonnegative functions $u$ and $v$ with the properties in \eqref{solu:property} and \eqref{-4.15}-\eqref{-4.17} follows from Lemmas \ref{lemma-4.1} and \ref{lemma-4.8} by a diagonal extraction procedure. Moreover, $u \geqslant 0$ by nonnegativity of all the $u_{\varepsilon}$. Based on \eqref{3.6-1} and \eqref{-4.16}, we obtain $v>0$ in $\overline \Omega \times[0, \infty)$.

We use \eqref{-3.37} with $p=l+1$ to deduce that
\begin{align}\label{-4.23}
\left(u_{\varepsilon}^l\right)_{\varepsilon \in(0,1)} \text { is bounded in } L^2\left((0, T) ; W^{1, 2}(\Omega)\right) \quad \text { for all } T>0. 
\end{align}
The regularity requirements \eqref{-2.1} and \eqref{-2.2} in Definition \ref{def-weak-sol} become straightforward consequences of \eqref{-4.15}, \eqref{-4.16} and \eqref{-4.23}. Given $\varphi \in C_0^{\infty}(\overline{\Omega} \times[0, \infty))$ satisfying  $\frac{\partial \varphi}{\partial \nu}=0$ on $\partial \Omega \times(0, \infty)$, using \eqref{sys-regul} we see that
\begin{align}\label{ident-1.10-1}
-\int_0^{\infty} \int_{\Omega} u_{\varepsilon} \varphi_t-\int_{\Omega} u_0 \varphi(0)=&-\frac{1}{2} \int_0^{\infty} \int_{\Omega} v_{\varepsilon}  \nabla u_{\varepsilon}^{2} \cdot \nabla \varphi
+ \int_0^{\infty} \int_{\Omega} u_{\varepsilon}^{2} v_{\varepsilon} \nabla v_{\varepsilon} \cdot \nabla \varphi \nonumber\\
& +\int_0^{\infty} \int_{\Omega} u_{\varepsilon} v_{\varepsilon} \varphi
\end{align}
for all $\varepsilon \in(0,1)$. We apply \eqref{-4.15} to show that
$$
-\int_0^{\infty} \int_{\Omega} u_{\varepsilon} \varphi_t \rightarrow-\int_0^{\infty} \int_{\Omega} u \varphi_t
$$
as $\varepsilon=\varepsilon_j \searrow 0$. Based on \eqref{-4.15}, \eqref{-4.16}, \eqref{-4.17} and \eqref{-4.23}, we have
$$
-\frac{1}{l} \int_0^{\infty} \int_{\Omega} v_{\varepsilon}  \nabla u_{\varepsilon}^{l} \cdot \nabla \varphi \rightarrow -\frac{1}{l} \int_0^{\infty} \int_{\Omega} v \nabla u^{l} \cdot \nabla \varphi
$$
and
$$
\int_0^{\infty} \int_{\Omega} u_{\varepsilon}^{l} v_{\varepsilon} \nabla v_{\varepsilon} \cdot \nabla \varphi \rightarrow  \int_0^{\infty} \int_{\Omega} u^{l} v\nabla v \cdot \nabla \varphi
$$
and
$$
\int_0^{\infty} \int_{\Omega} u_{\varepsilon} \varphi \rightarrow \int_0^{\infty} \int_{\Omega} u \varphi
$$
as well as
$$
\int_0^{\infty} \int_{\Omega} u^2_{\varepsilon} \varphi \rightarrow \int_0^{\infty} \int_{\Omega} u^2 \varphi
$$
as $\varepsilon=\varepsilon_j \searrow 0$. Therefore, \eqref{ident-1.10-1} implies \eqref{-2.3}. Similarly, \eqref{-2.4} can be verified.
\end{proof}
%%%%%%%%%%%%%%%%%%%%%%%%%%%%%%%%%%%%

%%%%%%%%%%%%%%%%%%%%%%%%%%%%%%%%%%%%%%%%%%%%%%%%%%%%%%%%%%%%%%%%%%%%%%%%%%%%%%%%%%%%%%%%

\section{Appendix: A key functional inequalities}

In this Appendix, we derive a functional inequality which play key roles in our analysis.  

%%%%%%%%%%%%%%%%%%%%%%%%%%%%%%%%%%%%%%%%%%%%%%%%%%%%%%%%%%%%%%

%%%%%%%%%%%%%%%%%%%%%%%%%%%%%%%%%%%%%%%%%%%%%%%%%%%%%%%%%%%%%%
\begin{lem}\label{lemma-3.5}
Let $\Omega \subset \mathbb{R}^2$ be a bounded domain with smooth boundary and $p \geqslant 1$. For each $\eta>0$ and any $\varphi, \psi \in C^1(\overline{\Omega})$ satisfying $\varphi,\psi>0$ in $\overline{\Omega}$, there holds
\begin{align}\label{eq-6.4}
\int_\Omega \varphi^{p+1} \psi |\nabla \psi|^2
\leqslant & \eta \int_\Omega \varphi^{p-1} \psi|\nabla \varphi|^2
  + c \left\{\left\|\psi \right\|^2_{L^{\infty}(\Omega)}+\frac{\left\|\psi \right\|^4_{L^{\infty}(\Omega)}}{\eta}\right\} 
    \cdot \int_{\Omega} \varphi^{p+1}   \cdot \int_\Omega \frac{|\nabla \psi|^4}{\psi^3} \nonumber\\
& + c \left\|\psi \right\|^2_{L^{\infty}(\Omega)} \cdot\left\{\int_\Omega \varphi\right\}^{2 p+1} 
    \cdot \int_\Omega \frac{|\nabla \psi|^4}{\psi^3} 
    + c \left\|\psi \right\|^2_{L^{\infty}(\Omega)} \cdot \int_{\Omega}\varphi \psi
\end{align}
for some constant $c=c(p)>0$.
\end{lem}

\begin{proof}
Recall the Sobolev imbedding inequality in $\Omega \subset \mathbb{R}^2$,
\begin{align}\label{eq-6.2}
\int_{\Omega} \rho^2 \leqslant C\|\nabla \rho\|_{L^1(\Omega)}^2 + C\|\rho\|_{L^{1}(\Omega)}^2, \quad 
  \rho \in W^{1,1}(\Omega).
\end{align}
Applying \eqref{eq-6.2}, we easily obtain
\begin{align}\label{eq-6.5}
\|\rho\|_{L^2(\Omega)} 
\leqslant C\|\nabla \rho\|_{L^1(\Omega)}+C\|\rho\|_{L^{\frac{1}{p+1}}(\Omega)}, \quad \rho \in W^{1,1}(\Omega)
\end{align}
for some $C=C(p)>0$. 
We use H\"{o}lder's inequality to yield
\begin{align*}%\label{-3.18}
\int_\Omega \varphi^{p+1} \psi |\nabla \psi|^2 \leqslant\left\{\int_\Omega \frac{|\nabla \psi|^4}{\psi^3}\right\}^{\frac{1}{2}} \cdot\left\{\int_\Omega \varphi^{2 (p+1)} \psi^{5}\right\}^{\frac{1}{2}}.
\end{align*}
For any $\varphi, \psi \in C^1(\overline{\Omega})$ satisfying $\varphi,\psi>0$ in $\overline{\Omega}$, we apply \eqref{eq-6.5} with $\rho=\varphi^{p+1} \psi^{\frac{5}{2}}$ to infer that
\begin{align*}%\label{-3.19}
\left\{\int_\Omega \varphi^{2 (p+1)} \psi^{5}\right\}^{\frac{1}{2}} 
= &\ \left\|\varphi^{p+1} \psi^{\frac{5}{2}}\right\|_{L^2(\Omega)} \nonumber\\
\leqslant &\  C \int_\Omega\left|(p+1) \varphi^{p} \psi^{\frac{5}{2}} \nabla \varphi
    +\frac{5}{2} \varphi^{p+1}\psi^\frac{3}{2}\nabla \psi\right|
    +C\left\{\int_\Omega \varphi \psi^{\frac{5}{2 (p+1)}}\right\}^{p+1} \nonumber\\
= &\ (p+1) C \int_\Omega \varphi^{p} \psi^\frac{5}{2}|\nabla \varphi|
     +\frac{5 C}{2} \int_\Omega \varphi^{p+1}\psi^\frac{3}{2}|\nabla \psi| \nonumber\\
  &\ + C \left\{\int_\Omega \varphi \psi^{\frac{5}{2 (p+1)}}\right\}^{p+1}.
\end{align*}
Applying H\"{o}lder's inequality, we show that
\begin{align*}
\int_G \varphi^{p} \psi^\frac{5}{2}|\nabla \varphi| 
\leqslant \left\|\psi\right\|_{L^{\infty}(\Omega)}^{2} \cdot\left\{\int_\Omega \varphi^{p+1}\right\}^{\frac{1}{2}} \cdot\left\{\int_\Omega \varphi^{p-1} \psi|\nabla \varphi|^2\right\}^{\frac{1}{2}}
\end{align*}
and
\begin{align*}
\int_\Omega \varphi^{p+1}\psi^\frac{3}{2}|\nabla \psi| 
\leqslant \left\|\psi\right\|_{L^{\infty}(\Omega)} \cdot\left\{\int_\Omega \varphi^{p+1}\right\}^{\frac{1}{2}} \cdot\left\{\int_\Omega \varphi^{p+1} \psi|\nabla \psi|^2\right\}^{\frac{1}{2}}
\end{align*}
as well as
\begin{align*}
\left\{\int_\Omega \varphi \psi^{\frac{5}{2 (p+1)}}\right\}^{p+1} 
& \leqslant \left\|\psi \right\|_{L^{\infty}(\Omega)}^2 \cdot\left\{\int_\Omega\varphi^{\frac{2 p+1}{2 (p+1)}}\cdot(\varphi \psi)^{\frac{1}{2 (p+1)}}\right\}^{p+1} \\
& \leqslant  \left\|\psi \right\|_{L^{\infty}(\Omega)}^2 \cdot\left\{\int_\Omega \varphi\right\}^{\frac{2 p+1}{2}} \cdot\left\{\int_\Omega \varphi \psi\right\}^{\frac{1}{2}}.
\end{align*}
Therefore, for all $\eta>0$, it follows that
\begin{align*}
\int_\Omega \varphi^{p+1} \psi |\nabla \psi|^2 
\leqslant & (p+1) C\left\|\psi\right\|_{L^{\infty}(\Omega)}^{2} \cdot\left\{\int_\Omega \frac{|\nabla 
    \psi|^4}{\psi^3}\right\}^{\frac{1}{2}} 
    \cdot\left\{\int_\Omega \varphi^{p+1}\right\}^{\frac{1}{2}} 
    \cdot\left\{\int_\Omega \varphi^{p-1} \psi|\nabla \varphi|^2\right\}^{\frac{1}{2}}\\
& +\frac{5 C\left\|\psi\right\|_{L^{\infty}(\Omega)}}{2} 
    \cdot \left\{\int_\Omega \frac{|\nabla \psi|^4}{\psi^3}\right\}^{\frac{1}{2}}
    \cdot\left\{\int_\Omega \varphi^{p+1}\right\}^{\frac{1}{2}} 
    \cdot\left\{\int_\Omega \varphi^{p+1} \psi|\nabla \psi|^2\right\}^{\frac{1}{2}} \\
& + C\left\|\psi \right\|_{L^{\infty}(\Omega)}^2 \cdot\left\{\int_\Omega \frac{|\nabla \psi|^4}{\psi^3}\right\}^{\frac{1}{2}} 
    \cdot\left\{\int_\Omega \varphi\right\}^{\frac{2 p+1}{2}} \cdot\left\{\int_\Omega \varphi \psi\right\}^{\frac{1}{2}} \\
\leqslant &   \frac{\eta}{2} \int_\Omega \varphi^{p-1} \psi|\nabla \varphi|^2
  +\frac{(p+1)^2 C^2\left\|\psi\right\|_{L^{\infty}(\Omega)}^4}{2 \eta} \cdot \int_\Omega \varphi^{p+1} 
    \cdot \int_\Omega \frac{|\nabla \psi|^4}{\psi^3} \\
& + \frac{1}{2} \int_\Omega \varphi^{p+1} \psi |\nabla \psi|^2
  + \frac{25 C^2\left\|\psi\right\|^2_{L^{\infty}(\Omega)}}{8} 
  \cdot \int_\Omega \varphi^{p+1}  \cdot \int_\Omega \frac{|\nabla \psi|^4}{\psi^3}\\
& +C\left\|\psi \right\|_{L^{\infty}(\Omega)}^2 \cdot\left\{\int_\Omega \varphi\right\}^{2 p+1} 
  \cdot \int_\Omega \frac{|\nabla \psi|^4}{\psi^3}
  +C\left\|\psi \right\|_{L^{\infty}(\Omega)}^2 \int_\Omega \varphi \psi
\end{align*}
which implies \eqref{eq-6.4} if we set $c=c(p)=\max \left\{2C, (p+1)^2 C^2, \frac{25 C^2}{4}\right\}$.
\end{proof}

%\hfill$ \Box$
%\bibliographystyle{siam}

%\bibliography{indrectsingball}

%\end{thebibliography}

\end{document}